\documentclass[12pt]{amsart}

\usepackage{amsmath,
amsthm,
amssymb,
mathrsfs,
amsfonts,
verbatim,
enumitem,
color,leftidx}
\usepackage{bbm}
\usepackage[all,tips]{xy}
\usepackage{hyperref}

\newtheorem{thm}{Theorem}[section]
\newtheorem{lem}[thm]{Lemma}
\newtheorem{observation}[thm]{Observation}
\newtheorem{claim}[thm]{Claim}
\newtheorem{cor}[thm]{Corollary}
\newtheorem{prop}[thm]{Proposition}

\newtheorem{conj}[thm]{Conjecture}

\theoremstyle{definition}
\newtheorem{de}[thm]{Definition}
\newtheorem{ex}[thm]{Example}

\theoremstyle{remark}
\newtheorem{remark}[thm]{Remark}
\newtheorem{discussion}[thm]{Discussion}

\numberwithin{equation}{section}

\makeatletter

\newcommand{\Rmnum}[1]{\expandafter\@slowromancap\romannumeral #1@}
\makeatother

\newcommand{\bp}{\mathbf{p}}
\newcommand{\bP}{\mathbf{P}}

\newcommand{\bu}{\mathbf{u}}

\newcommand{\bG}{\mathbf{G}}

\newcommand{\bN}{\mathbf{N}}
\newcommand{\bS}{\mathbf{S}}

\newcommand{\RR}{\mathbb{R}}

\newcommand{\NN}{\mathbb{N}}
\newcommand{\SL}{\operatorname{SL}}
\newcommand{\GL}{\operatorname{GL}}

\newcommand{\ZZ}{\mathbb{Z}}

\newcommand{\QQ}{\mathbb{Q}}

\DeclareMathOperator{\mult}{{\mathrm {mult}}}

\DeclareMathOperator{\diag}{\mathrm{diag}}
\DeclareMathOperator{\spa}{\mathrm{span}}
\newcommand{\Ad}{\operatorname{Ad}}

\DeclareMathOperator{\supp}{\mathrm{supp\,}}
\DeclareMathOperator{\rk}{\mathrm{rank}}
\DeclareMathOperator{\invdim}{\mathrm{invdim\,}}
\DeclareMathOperator{\stab}{\mathrm{stab}}
\DeclareMathOperator{\tot}{Tot}
\newcommand{\bra}{\left\langle}
\newcommand{\ket}{\right\rangle}

\newcommand{\conv}{{\rm conv}}

\newcommand{\fa}{\mathfrak{a}}

\newcommand{\fg}{\mathfrak{g}}

\newcommand{\fn}{\mathfrak{n}}

\newcommand{\ft}{\mathfrak{t}}
\newcommand{\fu}{\mathfrak{u}}
\newcommand{\fv}{\mathfrak{v}}

\newcommand{\fU}{\mathfrak{U}}

\newcommand{\cA}{\mathcal{A}}
\newcommand{\cB}{\mathcal{B}}
\newcommand{\norm}[1]{\left\|#1\right\|}

\def\calr{\mathcal{R}}

\def\rank{\operatorname{rank}}

\newcommand{\nin}{\notin}

\usepackage{enumitem}
\newlist{steps}{enumerate}{1}
\setlist[steps, 1]{label = Step \arabic*:}

\begin{document}

\title{On Topologically Big Divergent Trajectories}
\author{Omri N. Solan}
\address{Einstein Institute of Mathematics, Hebrew University, Israel}
\email{omrinisan.solan@mail.huji.ac.il}
\author{Nattalie Tamam}
\address{Department of Mathematics, University of Michigan, Ann Arbor, Michigan, MI 48109}
\email{nattalie@umich.edu}
\date{}

\begin{abstract}
We study the behavior of $A$-orbits in $G/\Gamma$, when $G$ is a semisimple real algebraic $\QQ$-group, $\Gamma$ is a non-uniform arithmetic lattice, and $A$ is a subgroup of the real torus of dimension $\ge\rk_\QQ G$. 
We show that every divergent trajectory of $A$  diverges due to a purely algebraic reason, solving a long-lasting conjecture of Weiss \cite[Conjecture 4.11]{Weiss}. 
In addition, we examine sets that intersect all $A$-orbits, and show that in many cases every $A$-orbit intersects every deformation retract $X\subset G/\Gamma$. 
This solves the questions raised by Pettet and Souto in \cite{PS}.
The proofs use algebraic and differential topology, as well as algebraic group theory. 
\end{abstract}
\maketitle
\markright{}
\section{Introduction} 
\label{sec: introduction}
Let $\bG$ be a semi-simple algebraic group over $\QQ$, $G := \bG(\RR)$, and $\Gamma := \bG(\ZZ)$.
Orbits in $G/\Gamma$ have been extensively studied in the last four decades. Many of the results were motivated by their strong connection to various problems in number theory (see \cite{dani, Littlewood, number theory}). 
The study can be roughly divided into two parts. The first part - unipotent orbits, were fully classified in the landmark result of Ratner \cite{ratner}, and behave rather tamely, with known orbit closures and invariant measures. 
The second part - diagonal orbits, seem to behave wildly, with some remarkable phenomenons, see \cite{Maucourant, shapira, T} for surprising counter-intuitive examples, and \cite{Lindenstrauss notes} for a recent state-of-the-art survey. 

In the present work, we study the orbits of diagonal subgroups of $G$ of a `large enough' dimension. 
More explicitly, let $T$ be a maximal $\RR$-torus and $A\subset T$ be a real subgroup which satisfies $\rk_\QQ G\le \dim A \le \rk_\RR G$.
We show that the large dimension of $A$ can be exploited in order to describe the divergent trajectories under the action of $A$. In particular, we show that there are finitely many `reasons' for the divergence of such orbits, which we call obvious (see the definition of obvious divergent trajectory in \S\ref{sec: divergence}). 
Thus, showing the following classification of divergent trajectories, which was conjectured by Weiss \cite{Weiss}, for the action of $A\subset T$ when $G$ is an almost $\QQ$-simple algebraic Lie group:
\begin{itemize}
    \item If $\dim A>\rank_{\QQ}G$, then there are no divergent $A$-orbits. 
    \item If $\dim A = \rank_{\QQ} G$, then the only divergent $A$-orbits are the obvious ones (See Definition \ref{def: obvious} of obvious divergence).
    \item If $\dim A < \rank_{\QQ} G$, then there are non-obvious divergent $A$-orbits. 
\end{itemize}
The first and the third of the above claims were shown in \cite{Weiss2} and \cite{tamam2}, respectively.  
Note that one may deduce a classification for a general algebraic $\QQ$-semisimple Lie group, but a bit more care is needed (see \cite{tamam2}). 

The result we present here generalizes Tomanov and Weiss \cite{TW}, proving the second of the above claims assuming $\rank_\QQ G=\rank_\RR G$. The results in \cite{TW} were also generalized to the $S$-adic setting in \cite{T S-adic}. The second of the above claims was also shown by Hattori \cite[Theorems A, B]{Hattori} when $\rank_\QQ G=1$. 

Our methods generalize Weiss' proof of the first part of the above claim \cite[Corollary]{Weiss2}, see Remark \ref{rem: Barak's proof}. We utilized a similar mythology to that of Weiss, however the existence of divergent trajectories and the attempt to describe them case additional challenges, which we need to address.
A crucial step in \cite{TW} (and a result interesting in its own right) was showing that there is a fixed compact subset of $G/\Gamma$ which intersects every diagonal orbit. This result is false in general for $A$-orbits satisfying $\rk_\QQ(\bG)\le \dim A < \rk_\RR(\bG)$, as was pointed out by Tomanov and Weiss \cite[Example 1]{TW}. In contrast to their example, we will show that for many such subgroups $A$, every $A$-orbit intersects any 
deformation retract of $G/\Gamma$.  

\subsection{Divergence}\label{sec: divergence}

Let $\pi$ be the projection $G\rightarrow G/\Gamma$ defined by $\pi(g):=g\Gamma$ for any $g\in G$.

\begin{de}[Divergence]\label{def: divergence}
We say that an orbit $A\pi(g)$ \emph{diverges} if for every compact set $K\subset G/\Gamma$ there is a compact set $K_A\subset A$ such that $a\pi(g) \nin K$ for every $a\in A\setminus K_A$. 
\end{de}

As was first shown by Dani \cite{dani}, divergent trajectories in some homogeneous dynamical systems are in correspondence with singular vectors (or systems of linear forms), i.e., vectors for which the Dirichlet theorem can be infinitely improved. This correspondence motivated a lot of the research on divergent trajectories, see, e.g., \cite{C2, singular system}.
Since rational vectors are singular in a somewhat `obvious' way, this connection also motivated Dani \cite{dani} and Weiss \cite{Weiss} to distinguish between `obvious' divergent and `non-obvious' trajectories. 

\begin{de}[Obvious divergence]\label{def: obvious}
A trajectory $A\pi(g)$ is said to diverge \emph{obviously} if there exist finitely many rational representations $\varrho_1,\dots,\varrho_k$ and vectors $v_1,\dots,v_k$, where  $\varrho_j: G\rightarrow\GL(V_j)$ and $v_j\in V_j(\QQ)$, such that for any divergent sequence $(a_i)_{i=1}^\infty\subset A$ there exist a subsequence $(a_i')_{i=1}^\infty\subset (a_i)_{i=1}^\infty$ and an index $1\le j\le k$, such that $\varrho_j(a_i'g)v_j\xrightarrow{i\to \infty} 0$. 
\end{de}

It was proved in \cite{Weiss} that obvious divergence indeed implies standard divergence.
We show that if the dimension of $A$ is big enough, then the only divergent trajectories are the obviously divergent ones. 

\begin{thm}\label{thm: divergence classification}
Assume $\dim A= \rk_\QQ(\bG)$ and let $A\pi(g)$ be a divergent trajectory. Then $A\pi(g)$ diverges obviously.
\end{thm}

\begin{remark}[These divergent trajectories are in fact degenerate]
In the proof of Theorem \ref{thm: divergence classification} we get the more restrictive description of the divergent trajectories: All the `shrinking' vectors coming from Definition \ref{def: obvious} are conjugates of highest weight vectors, and the corresponding representations are $\QQ$-fundamental representations (see Observation \ref{obs: fundamental representations} for the definition of $\QQ$-fundamental representations). In particular, an orbit $A\pi(g)$, as in Theorem \ref{thm: divergence classification}, is a degenerate divergent trajectory, as defined by Dani in \cite{dani} and generalized by Weiss in \cite{Weiss}. 
\end{remark}

Theorem \ref{thm: divergence classification} solves a conjecture of Weiss \cite[Conjecture 4.11]{Weiss}. Other parts of this conjecture were shown in \cite{tamam2, TW, Weiss2}. 

In \cite{TW} Tomanov and Weiss found a simple algebraic description for all divergent trajectories in the case $A=T$. Such a simple description can not be true in general, as pointed out in \cite[Example 2]{TW}. Here, we present another example showing that any algebraic description in our setting ought to be more complicated. 
\subsection{Set intersection}
The problem of set intersection is the study of finding sets that intersect every orbit of a certain action. Results of this kind can be seen in \cite[Theorem 1.3]{Mc} and \cite[Theorem 1.3]{TW}.

In order to emphasize the strength of the
techniques which we use in order to prove Theorem \ref{thm: divergence classification}, we prove an additional theorem when $\dim A = \rk_\RR(\bG)$ and revisit an example where $\dim A = \rk_\QQ(\bG) < \rk_\RR(\bG)$. 

\begin{de}
A subspace $X_0$ of $X$ is called a \emph{deformation retract} of $X$ if there is a homotopy $F: X\times [0,1]\rightarrow X$ such that for all $x\in X$ and $y\in X_0$,\[
F(x,0)=x,\quad F(x,1)\in X_0,\text{ and}\quad F(y,1)=y.\]
Such a homotopy $F$ is called a \emph{deformation retraction}.

Given $Y, Z \subset X$, we say that $Y$ can be \emph{homotoped away from $Z$} if $Y$ can be deformed via a homotopy so that it does not intersect $Z$.
\end{de}

\begin{thm}\label{thm: intersection}
Assume $\dim A = \rk_\RR(\bG)$.
Let $X_0\subset G/\Gamma$ be a deformation retract and let $x\in G/\Gamma$. Then $Ax\cap X_0\neq\emptyset$.
Moreover, for any bounded map $f: A\to G$ there exists $a\in A$ such that $f(a)ax\in X_0$.
\end{thm}
\begin{remark}
The concepts of bounded map and deformation retract can be generalized. The map $f$ can be replaced by a bounded correspondence of non-zero multiplicity, see Definition \ref{def: correspondence}. The deformation retract can be replaced by the image of a homotopy equivalence $X\to G/\Gamma$. 
\end{remark}
\begin{ex}
There are many examples of deformation retracts to arithmetic homogeneous spaces $G/\Gamma$, and we present some of the known ones here.
The set of well-rounded lattices in $\SL_n(\RR)/\SL_n(\ZZ)$ is a deformation retract. This example can be further extended to the semisimple part of $\GL_n(D)$ for some division algebra $D$ over $\QQ$, see Ash \cite{Ash}. Such deformation retracts are of minimal dimension.
The set of stable lattices in $\SL_n(\RR)/\SL_n(\ZZ)$ (sometimes called semi-stable) is a much larger deformation retract, and the notion can be extended to all arithmetic homogeneous spaces $G/\Gamma$, see Grayson \cite{gr1} for an exposition on stable lattices and \cite{gr2} for the general case.
Another explicit construction in the general case is given by Saper \cite{saper}.
All of these examples are formulated as deformation retracts of the quotient of the symmetric space $X/\Gamma$, where $X = K\backslash G$ and $K$ is the maximal compact subgroup, but the retraction can be lifted to a retraction of $G/\Gamma$.
\end{ex}

By adding a certain condition one can extend Theorem \ref{thm: intersection} to the case when $\rk_\QQ(\bG)\le \dim A < \rk_\RR(\bG)$.
The theorem does not hold for all subgroups $A$ (see Example \ref{ex: compact peripheral orbit}), only for a Zariski open set of them, in the following sense.
Let ${\rm Gr}(\ft,l)$ be the set of all $l$-dimensional subspaces $\fa \subset \ft$, where $\ft$ is the Lie algebra of the torus $T$.
The Grassmannian ${\rm Gr}(\ft,l)$ is a real algebraic variety. 

\begin{thm}\label{thm: alternative intersection}
Let $X_0\subset G/\Gamma$ be a deformation retract, let $l\ge \rk_\QQ(\bG)$, and let $Tx\subset G/\Gamma$ be a trajectory. 
Then there is a nonempty Zariski open subset $U\subset {\rm Gr}(\ft,l)$, which depends only on $\bG$ and $l$, such that if $\operatorname{Lie}(A)\in U$, then the conclusion of Theorem \ref{thm: intersection} holds.
Specifically, $U$ is the set of Lie algebras of subgroups $A$ with good restrictions with respect to $\Psi$, where $\Psi$ is defined as in Proposition \ref{prop: cover locally finite and char lin indp} and the good restrictions property is defined as in Definition \ref{de: has good restriction}. 
\end{thm}
Theorem \ref{thm: alternative intersection} generalizes the results in \cite{Mc, S} for arbitrary homogeneous spaces and arbitrary deformation retract, instead of only the sets of well-rounded and stable lattices, as well as a result in \cite{PS}, which shows that compact $A$-orbits cannot be homotoped away from compact sets. In particular, it answers \cite[Questions $2,3$]{PS} positively.
\cite[Question $1$]{PS} is also answered by Theorem \ref{thm: alternative intersection} via the next corollary. 
\begin{cor}\label{cor: un-homotopable compact orbit}
Let $\Gamma, T\subset G$ be as before. 
There exist a $\rk_\QQ(G)$-dimensional subgroup $A\subset T$ and a compact trajectory $Ax\subset G/\Gamma$ which cannot be homotoped away from compact sets. 
\end{cor}
Corollary \ref{cor: un-homotopable compact orbit} can not be extended to every compact \linebreak $\rk_\QQ(G)$-dimensional trajectory, as shown by the following example, originally introduced in \cite[Example 1]{TW}, although without the compactness of the orbit $Ax$.
\begin{ex}
\label{ex: compact peripheral orbit}
There exists an algebraic group $\bG$ with \linebreak $\rk_\QQ(\bG)=1 < \rk_\RR(\bG)=3$, such that for the maximal $\QQ$-split $\QQ$-torus in $G$, denoted $S$, and some $x\in G/\Gamma$ the following holds. 
The orbit $Sx$ is compact and can be homotoped away from compact sets using group elements; that is, for every compact set $K\subset G/\Gamma$ there is $g\in G$ such that $gSx\cap K = \emptyset$. 
\end{ex}
Example \ref{ex: compact peripheral orbit} explicitly described in \S \ref{sec: counter_example}. 

The following result is obtained by combining the arguments used in the proofs of Theorems \ref{thm: divergence classification} and \ref{thm: alternative intersection}:
\begin{prop}\label{prop: deformation retract intersection}
    An $A$-orbit $A\pi(g)$ intersects any deformation retract if one of the following conditions is satisfied:
    \begin{itemize}
        \item $A\pi(g)$ diverges.
        \item 
        For any $\QQ$-fundamental representation $\varrho: G\rightarrow \GL(V)$, any $v\in V(\QQ)$ which equals $\bp_\fv$ for some unipotent radical $\fv$, and any unbounded sequence $\{a_n\}\in A$, the set $\{\|\varrho(a_ng)v\|\}$ is also unbounded, see \S\ref{sec: compactness criterion} for a discussion on these objects.
    \end{itemize}
\end{prop}
The ideas of the proof can be found in Remark \ref{rem: proof of cor somthing}.

\subsection{Further research} 
\label{sub: further_research}
Although we provide a characterization of divergent trajectories, it is interesting to show a more quantitative result.
\begin{conj}
For every compact sets $K\subset K_1\subset G/\Gamma$ there is a compact set $K_2\subset G/\Gamma$ with the following property. 
If a trajectory $Ak$, $k\in K$, eventually exits $K_2$, then it exits $K_1$ for finitely many representational witnesses.
That is, if the set $\{a\in A: ak\Gamma\in K_2\}$ is compact, then there exist a finite set of rational representations $\varrho_1,\dots,\varrho_m$, rational vectors $v_1,\dots,v_m$, and open subsets $U_1,\dots, U_m$, where $\varrho_j: G\rightarrow\GL(V_j)$, $v_j\in V_j(\ZZ)$, and $0\in U_j\subset V_j$, such that:
\begin{itemize}
    \item 
    for all $1\le j\le m$, $g\in G$, \[\varrho_j(g)v_j\in U_j \implies \pi(g)\nin K_1, \text{ and} \]
    \item the set 
    $\{a\in A:\forall j. ~ \varrho_j(ak)v_j\nin U_j\}$ is compact.
\end{itemize}
\end{conj}
A stronger version of the above conjecture holds for the case $\rank_\QQ(G) = \rank_\RR(R)$, see \cite[Theorem 1.3]{TW}. We believe that using the methods presented in this manuscript one should be able to prove it in the setting of Theorem \ref{thm: alternative intersection}, that is, if $A\subseteq T$ has good restrictions (see Definition \ref{de: has good restriction}). 

Another task is to provide a better classification of the divergent trajectories, namely, determine which vectors in the fundamental representation are needed in Theorem \ref{thm: divergence classification}.
\begin{conj}\label{conj: divergent trajectories}
Let $S$ denote the rational torus and $P_1,..., P_r$ the different maximal parabolic subgroups containing $S$, and $\fv_1,...,\fv_r$ the corresponding Lie algebras of the unipotent radicals. Then, for every $A$ with $\dim A = \rk_\QQ(\bG)$ and any divergent trajectory $A\pi(g)$ for $g\in G$, there exists $g_\QQ\in \bG(\QQ)$ such that the orbit $A\pi(g)$ satisfies the obvious divergence property using the vectors $\varrho_1(g_\QQ) \bp_{\fv_1},\dots,\varrho_r(g_\QQ) \bp_{\fv_r}$  (see \S\ref{sec: compactness criterion} for the definitions of $\bp_\fv$ for a unipotent radical Lie-algebra $\fv$ and the representations $\varrho_1,\dots,\varrho_r$ in this setting). 
\end{conj}

In \S\ref{sec: counter_example} we show that Conjecture \ref{conj: divergent trajectories} holds for a specific example, see Proposition \ref{prop: example divergent classification}. 


\subsection{Overview of the paper}
The main theorems are shown by using a construction of a cover for $A$ (as a manifold) and by analyzing its possible properties. 
We use techniques from different mathematical fields.
\begin{itemize}
\item {\bf Algebraic group theory:}
The goal of \S\ref{sec: bounded Bruhat}-\ref{sec: the cover} is to construct a cover of $A$ given a trajectory.
This cover is constructed in \S\ref{sec: the cover}, using a compactness criterion shown in \cite{TW}, which is discussed in \S \ref{sec: compactness criterion}. 
In \S\ref{sec: the cover} we also show some properties of the cover: its covering number and an algebraic description of it. 
A main tool in obtaining an algebraic description of the sets in the cover is Theorem \ref{thm: bounded Bruhat decomp}, which is a bounded Bruhat-type theorem and is proved in \S \ref{sec: bounded Bruhat}. 
\item {\bf Homological algebra:}
In \S \ref{sub: covering_theorem} we disprove the existence of certain coverings of manifolds with low covering number. This part is mostly independent of the rest of the paper. The main result of this section is Theorem \ref{thm: covering}, which is proved using cohomology theory. We analyze a certain nontrivial cycle and find an equivalent cycle that is supported on the intersection of many open sets. 
The technique is a direct generalization of the topological tools given at used in \cite{S}, themselves a simplification of the topological tools used in \cite{Mc}.
\item {\bf Geometry:}
In \S \ref{sec: bordered sets} we analyze a certain kind of shapes that constitute the cover corresponding to divergent trajectories. 
We show that they are contractible or empty even if we subtract an arbitrarily large bounded set.
\item {\bf Differential topology:}
In \S\ref{sec: proof_of_theorem_thm: intersection} we prove Theorem \ref{thm: alternative intersection} by deforming the configuration we have. We start with an orbit-like set  $\{f(a)ax: a\in A\}$, for some bounded correspondence $f: A\to G$, which does not intersect a given deformation retract $X_0$. We use it to find another orbit-like set $\{\tilde{f}(a)ax: a\in A\}$ which does not intersect a potentially much larger set $K_0$ which appears naturally in the cover theorem. Then, using the latter, we get a contradiction. 
\end{itemize}

Theorems \ref{thm: divergence classification} and \ref{thm: alternative intersection} are proven in \S\ref{sec: proof of divergence classification} and \S\ref{sec: proof_of_theorem_thm: intersection}, respectively.

In \S\ref{sec: counter_example} we discuss \cite[Example 1]{TW}, which shows that the consequences of Theorem \ref{thm: intersection} do not hold when $\rk_\QQ(\bG)<\rk_\RR(\bG)$.

\subsection*{Acknowledgment}
We would like to thank B. Weiss for familiarizing us with the problem and for suggesting that we should team up to tack it. 
The first author would like to thank his father E. Solan for helpful discussions.
The second author would also like to thank R. Spatzier for his interest in this project and for helpful discussions. 
We also want to thank the anonymous referees for several important suggestions. 
This work is part of the Ph.D. thesis of the second author.

\section{\label{sec: bounded Bruhat} A bounded Bruhat type decomposition}

All square matrices admit an $LU$ factorization with partial pivoting. That is, for any square matrix $A$ there exists a permutation matrix $P$, a lower triangular unipotent matrix $L$ with all entries bounded by $1$, and an upper triangular matrix $U$, such that $A=PLU$ (see \cite[Lecture 21]{LU decom}). The goal of this section is to prove Theorem \ref{thm: bounded Bruhat decomp}, which is an analog of this factorization for a general Lie group, using the Bruhat decomposition. 

We use standard notation from the theory of linear algebraic groups, see \cite{borel,knapp}. 

Let $\Delta_{\RR}$ be an $\RR$-simple system of the $\RR$-root system $\Phi_\RR$ of $G$. 
Let $\Phi^+_\RR$ be the set of positive $\RR$-roots defined by $\Delta_{\RR}$.
For $\lambda\in\Phi_{\RR}$, denote by $\fg_{\lambda}$ the $\RR$-root space for $\lambda$. 

Let \begin{equation}\label{eq: N,B defn}
	\fn:=\bigoplus_{\lambda\in\Phi_{\RR}^+}\fg_{\lambda},\quad N:=\exp(\fn),\quad B:=N_G(N).
\end{equation}
Then, $B$ is a Borel subgroup.

Let $W(\Phi_{\RR})$ be the $\RR$-Weyl group of $G$. 
According to \cite[\S VI.5]{knapp} $W(\Phi_{\RR})$ acts simply transitively on $\RR$-simple systems. 
In particular, there exists a unique $w_0\in W(\Phi_{\RR})$ such that 
\begin{equation}\label{eq: w0 defn}
	w_0(\Delta_{\RR})=-\Delta_{\RR}. 
\end{equation}
By \cite[\S VI.5]{knapp}, the Weyl group satisfies $W\left(\Phi_{\RR}\right)\cong N_{G}\left(T\right)/Z_{G}\left(T\right)$.
For every $w\in W\left(\Phi_{\RR}\right)$, let $\bar{w}$ be a representative of $w$ in $N_{G}\left(T\right)$. 

\begin{thm}\label{thm: bounded Bruhat decomp}
	There exists a compact set $N_0\subset N$ such that $$G=W(\Phi_{\RR})N_0\bar{w}_0B.$$  
\end{thm}

Let $\varrho: G\rightarrow\GL(V)$ be an $\RR$-highest weight representation, i.e., an $\RR$-representation with highest weight defined over $\RR$. Denote the $\RR$-highest weight of $\varrho$ by $\chi$. There is a direct sum decomposition 
\begin{equation}
	V=\bigoplus_{\lambda\in\Phi_{\varrho}}V_{\lambda},\label{eq: V rho decomp}
\end{equation}
where $\Phi_{\varrho}$ is the set of $\RR$-weights for $\varrho$, and for any $\lambda\in\Phi_{\varrho}$, $V_{\lambda}$ is the $\RR$-weight vector space for $\lambda$. For any $\lambda\in\Phi_{\varrho}$ let $\varphi_{\lambda}: V\rightarrow V_{\lambda}$ be the projection associated with the decomposition \eqref{eq: V rho decomp}. Let $\Phi_{\varrho}^+$ be the set of positive $\RR$-weights for $\varrho$, where the order is defined with $\Delta_{\RR}$. 

For any $\lambda\in\Phi_{\varrho}$ and $w\in W\left(\Phi_{\RR}\right)$ we have 
\begin{equation}\label{eq: weyl act weight}
	\varrho\left(\bar{w}\right)V_{\lambda}=V_{w\left(\lambda\right)}.
\end{equation}
The following is a useful corollary of Theorem \ref{thm: bounded Bruhat decomp}.

\begin{cor}
	\label{cor: weyl elem norm bound} 
	For any $\RR$-highest weight representation $\varrho: G\rightarrow\GL(V)$ with highest weight $\chi$ and a choice of norm $\norm{\cdot}$ on $V$, there exists $c>0$ which satisfies the following. 
	For any $h\in G$ there exists $w\in W(\Phi_{\RR})$, which only depends on $h$ (and not on $\varrho$), such that for any $v\in V_{\chi}$
	\begin{equation*}
		\norm{\varrho\left(h\right)v}\leq c\cdot\norm{\varphi_{w\left(\chi\right)}\left(\varrho\left(h\right)v\right)}.
	\end{equation*}
\end{cor}
In other words, the size of a vector $\varrho\left(h\right)v$ is controlled by its components corresponding to the weights $w\left(\chi\right)$ for $w\in W(\Phi_\RR)$. Corollary \ref{cor: weyl elem norm bound} is a stronger version of \cite[Theorem 3.1]{tamam}, as the element $w$ does not depend on the representation $\varrho$. This is used in the proof of Proposition \ref{prop: cover locally finite and char lin indp}(2) to show the linear independence of the characters.

\begin{proof}[Proof of Corollary \emph{\ref{cor: weyl elem norm bound}} assuming Theorem \emph{\ref{thm: bounded Bruhat decomp}}]
	Let $h\in G$. By Theorem \ref{thm: bounded Bruhat decomp}, there exists a decomposition $h=\bar{w}n\bar{w}_0b$, where $w\in W(\Phi_{\RR})$, $n\in N_0$, $w_0$ is as in Eq. \eqref{eq: w0 defn}, $b\in B$, and $N_0$ is a fixed compact subset of $N$. 
	
	Since $B$ stabilize the $\RR$-highest weight vector space, $\varrho(b)v\in V_{\chi}$. By Eq. \eqref{eq: weyl act weight}, we may deduce that $\varrho(\bar{w}_0b)v=u\in V_{w_0(\chi)}$. Since $n\in N$, we have $\varrho(n)u-u\in\bigoplus_{\lambda> w_0(\chi)}V_{\lambda}$.
	Therefore, 
	\begin{align}\label{eq: proj is u}
		\varphi_{ww_0(\chi)}(\varrho(h)v)&=\varphi_{ww_0(\chi)}(\varrho(\overline{w}n)u)\varphi_{w_0(\chi)}(\varrho(n)u)=u
	\end{align}
	Since $N_0$ is compact,  
	\begin{equation}\label{eq: norm closeto u}
		\norm{\varrho(h)v}\leq c\norm{u},
	\end{equation}
	for some constant depending only on the choice of the norm on $V$. 
	The claim now follows from \eqref{eq: proj is u} and \eqref{eq: norm closeto u}. 
\end{proof}

The rest of this section is devoted to proving Theorem \ref{thm: bounded Bruhat decomp}. We will use the following real version of Bruhat decomposition, see \cite[\S 14.15]{borel}. 
\begin{thm}[Bruhat decomposition]\label{thm: bruhat}
For every semisimple real group $G$ we have
$$G = \biguplus_{w\in W\left(\Phi_{\RR}\right)}N\bar{w}B.$$
\end{thm}
We will enlarge each cell in the decomposition to get an open cover 
\begin{align}\label{eq: cover bruhat}
G=\bigcup_{w\in W(\Phi_{\RR})}\bar{w}N\bar{w}_0B.
\end{align}
Then we will shrink each open set of the cover \eqref{eq: cover bruhat} to obtain the desired closed cover.

\begin{proof}[Proof of Theorem \emph{\ref{thm: bounded Bruhat decomp}}]
Consider the Lie group $N_w := \bar{w}^{-1}N\bar{w}$. Its Lie algebra is $\fn_w := \bigoplus_{\lambda\in \Phi_{\RR}^+}\fg_{w(\lambda)}$. 
It follows that $\fn_w = (\fn_w\cap \fn)\oplus (\fn_{w}\cap \fn_{w_0})$, that is, the Lie algebras of $N_w\cap N, N_{w}\cap N_{w_0}$ span the Lie algebra of $N_w$.
It follows from \cite[\S 14.4]{borel} that \[N_w = (N_w\cap N_{w_0}) \cdot (N_w\cap N). \]
Hence, using Bruhat decomposition, 
\begin{align*}
	G=&\biguplus_{w\in W\left(\Phi_{\RR}\right)}\bar{w}N_wB = 
	\biguplus_{w\in W\left(\Phi_{\RR}\right)}\bar{w}(N_w\cap N_{w_0}) \cdot (N_w\cap N)B
	\\&\subset 
	\bigcup_{w\in W\left(\Phi_{\RR}\right)}\bar{w}N_{w_0} \cdot NB
	=
	\bigcup_{w\in W\left(\Phi_{\RR}\right)}\bar{w}\bar w_0^{-1} N \bar w_0 \cdot B.
\end{align*}
Therefore,  
\begin{align}\label{eq: semi bounded bruhat}
	\bigcup_{w\in W\left(\Phi_{\RR}\right)} \bar w N\bar w_0B = G.
\end{align}

Let $p_0\in G/B$ denote the trivial coset.
Then, \eqref{eq: semi bounded bruhat} is equivalent to 
\begin{align}\label{eq: cover of compact}
	\bigcup_{w\in W\left(\Phi_{\RR}\right)}\bar{w}N \bar w_0 p_0 = G/B.
\end{align}
The orbit $N \bar w_0 p_0$ must have a nontrivial interior as it is constructible and the union of its translates covers $G/B$. 
Since it is an orbit, it is open. Thus, \eqref{eq: cover of compact} is an open cover.
	
Since \eqref{eq: cover of compact} is an open cover, there exists a closed sub-cover $\{V_ww\in W(\Phi_\RR)\}$ for $V_w\subset \bar{w}N \bar w_0 p_0$. 
Since $G/B$ is compact, the set $V_w$ must be compact. 
Since $N_{w_0}\cap B = \{e\}$ (see \cite[\S 14.1]{borel}), the map $N\to N\bar w_0 p_0$ is a homeomorphism.
Hence, there are compact subsets $N_{0,w}\subset N$ such that $V_w = \bar{w}N_{0,w} \bar w_0 p_0$. It remains to set $N_0 := \bigcup_{w\in W\left(\Phi_{\RR}\right)}N_{0,w}$.
\end{proof}

\section{\label{sec: compactness criterion} A compactness criterion}
In this section, we recall a compactness criterion obtained by Tomanov and Weiss in \cite{TW} and further developed by Kleinbock and Weiss in \cite{KleinbockWeiss}. 

According to \cite[Theorem 3.4]{borel}, the Lie algebra of $\textbf{G}$ is equipped with a $\QQ$-structure which is compatible with the $\QQ$-structure of $\textbf{G}$. Let $\fg:=\operatorname{Lie}(\textbf{G})(\RR)$ and $\fg_\ZZ:=\operatorname{Lie}(\textbf{G})(\ZZ)$.

Let $P_1,\dots, P_r$ be the maximal  $\QQ$-parabolic subgroups containing a fixed minimal $\QQ$-parabolic subgroup, and let $\fu_1,\dots, \fu_r$ denote the Lie algebras of their unipotent radicals. Then, $r =\operatorname{rank}_\QQ(\bG)$ (see \cite{borel}).
For $j = 1,\dots, r$,  
let $\calr_j$ denote the set of all the Lie algebras of unipotent radicals of conjugates of $P_j$ defined over $\QQ$. Set  $\calr:=\bigcup_{j}\calr_j$.

\begin{de}
	Given a neighborhood $W$
	 of $0$ in $\fg$, and 
	 $g \in G$, an element $\fu\in\calr$ is called 
	 \emph{$W$-active for $g$} if 
	 $\Ad(g)\fu\subset\spa(W\cap\Ad(g)\fg_\ZZ)$.
\end{de}

\begin{prop}[{Compactness criterion \cite[Proposition 3.5]{TW}}] \label{prop: compactness criterion}
	For any $L\subset G,$ $\pi(L)\subset G/\Gamma$ is unbounded if and only if for any neighborhood $W$ of $0$ in $\fg$ there exist $g\in L$ and $\fu \in\calr$ which is $W$-active for $g$.
\end{prop}

In other words, given a sequence $(g_i)_{i=0}^\infty\subset G$, the sequence $(\pi(g_i))_{i=0}^\infty$ diverges if and only if for every neighborhood $W$ of $0$ in $\fg$ there is $i_0\ge0$ such that for every $i\ge i_0$ there is $\fu \in\calr$ which is $W$-active for $g_i$.

A linear subspace of $\fg$ is called \emph{unipotent} if it is contained in the Lie algebra of a unipotent subgroup. Note that a sub-Lie algebra of $\fg$ is unipotent if and only if it is the Lie algebra of some unipotent subgroup.  

\begin{prop}\label{prop: W defn}{\cite[Proposition 3.3]{TW}}
	There is a neighborhood $W_0$ of $0$ in $\fg$ such that for any $g\in G,$ the span of $\Ad(g)\fg_\ZZ\cap W_0$ is unipotent.
\end{prop}

\begin{prop}\label{prop: at most r intersections}{\cite[Proposition 3.5]{KleinbockWeiss}}
	Suppose that for some $j\in \{1,\dots,r\}$ and $\fu,\fu'\in\calr_j$, the subspace $\operatorname{span}(\fu,\fu')$ is unipotent. Then $\fu = \fu'$. In particular, for any unipotent subspace $\fv\subset\fg$, \[\#\{\fu\in\calr\::\:\fu\subset\fv\}\le r.\]
\end{prop}

\begin{cor}\label{cor: one conjugating elemnt for intersections}
	Suppose that for some $1\le j_1<\dots<j_m\le r$
	and $\fv_1\in\calr_{j_1},\dots,\fv_m\in\calr_{j_m}$, the subspace  $\operatorname{span}(\fv_1,\dots,\fv_m)$ is unipotent. 
	Then, there exists $h\in G$ such that $\fv_i=\Ad(h)\fu_{j_i}$ for all $1\le i\le m$.
\end{cor}

\begin{proof}
    Since $\operatorname{span}(\fv_1,\dots,\fv_m)$ is unipotent and defined over $\QQ$, it is contained in a maximal unipotent subgroup defined over $\QQ$, which is the unipotent radical of a minimal $\QQ$-parabolic group, i.e., conjugated to a subspace of $\fn$. Assume $h\in G$ satisfies \[
	\spa(\fv_1,\dots,\fv_m)\subset\Ad(h)\fn. \]
	Then, for each $1\le i\le m$ the sets $\Ad(h)\fu_{j_i}$, $\fv_i$ are both in $\calr_i$ and   $\operatorname{span}(\Ad(h)\fv_i,\fu_i)$ is unipotent. Hence, Proposition \ref{prop: at most r intersections} implies that $\Ad(h)\fv_i=\fu_i$. 
\end{proof}

For $j = 1,\dots,r$ and $\fu\in\calr_j$, let $\bp_\fu = \bu_1\wedge\cdots\wedge\bu_{d_j} \in \tilde{V}_j:=\bigwedge^{d_j}\fg$, where $\bu_1,\dots,\bu_{d_j} \in\fg_\ZZ$ form a basis for the $\ZZ$-module $\fu\cap\fg_\ZZ$ ($\bp_\fu$ is uniquely determined up to a sign). Let\[\tilde{\varrho_j}:=\bigwedge^{d_j}\Ad: G\rightarrow \GL(\tilde{V}_j).\]
For $j=1,\dots,r$, let $V_j:=\operatorname{span}(\varrho_j(G)\bp_{\fu_j})$, and let $\varrho_j$ be the restriction of $\tilde{\varrho}_j$ to $\GL(V_j)$.
Since all the elements of $\calr_j$ are  conjugates of $\fu_j$, we have \[\left\{\bp_\fu\::\:\fu\in\calr_j\right\}\subset V_j. \]

\begin{observation}\label{obs: fundamental representations}
For each $1\le j\le r$, the space spanned by $\bp_{\fu_j}$ is fixed by a parabolic subgroup of $G$. Hence, the representations $\varrho_1,\dots,\varrho_r$ are $\QQ$-highest weight representations, and in particular, irreducible. 
Denote the highest weight of $\varrho_i$ by $\chi_i$, $1\le i\le r$, and let $\Delta_\QQ=\{\alpha_1,\dots,\alpha_r\}$ be a $\QQ$-simple system of $G$. Then, according to \cite[Lemma 5.1]{tamam2} for any $1\le i\le r$, we have 
\begin{equation}\label{eq: fundamental weights}
    \langle\chi_i,\alpha_j\rangle=c_i\delta_{ij}
\end{equation}
for some positive constant $c_i$, where the inner product is defined using the Killing form and $\delta_{ij}$ is the Kronecker delta. 
That is, $\varrho_1,\dots,\varrho_r$ are the \emph{$\QQ$-fundamental representations of $G$}. 
In addition, for each $1\le j\le r$, the vector $\bp_{\fu_j}$ is a highest weight vector for $\varrho_j$.

\end{observation}

\begin{prop}[{\cite[Corollary 3.3]{KleinbockWeiss}}]\label{prop: bound on vectors}
	For every $\varepsilon>0$ there exists a neighborhood $W_{\varepsilon}$ of $0$ in $\fg$ such that if $\fu\in\calr_j$, $1\le j\le r$, is $W_{\varepsilon}$-active for $g$, then
	\begin{equation*}\label{eq: small vectors representation}
	\norm{\varrho_j(g)\bp_\fu}<\varepsilon
	\end{equation*}
\end{prop}

\section{Construction of a cover}
\label{sec: the cover}
In this section, we construct a cover for every trajectory in $G/\Gamma$. The cover will encode the behavior of the trajectory near the cusps of $G/\Gamma$. To understand the cusps we will use the compactness criterion in Proposition \ref{prop: compactness criterion}.

Let $W_0$ be the intersection of the neighborhood of zero provided by Proposition \ref{prop: W defn} and the one provided by \ref{prop: bound on vectors} for $\varepsilon=1$.

let $K_0$ be the closure of the set of points $\pi(g)\in G/\Gamma$ which satisfy that all $\mathfrak{u}\in\mathcal{R}$ are not $W_0$-active
for $g$. According to Proposition \ref{prop: compactness criterion}, it is a bounded (and thus compact) set.
Then, the sets  
\begin{equation}\label{eq: cover defn}
	U_\fu:=\left\{g\in G\::\:\Ad(g)\fu\subset\operatorname{span}\left(\Ad(g)\fg_\ZZ\cap W_0 \right)\right\},\quad \fu\in\calr, 
\end{equation}
form an open cover of \[
G_0:=\left\{g\in G\::\:\pi(g)\notin K_0\right\}.\]
In this section, we describe some of the properties of this cover. 

The next lemma follows directly from Proposition \ref{prop: at most r intersections}.

\begin{lem}\label{lem: no large intersection} 
	Let $\fv_{1},\dots,\fv_{m}\in\calr$ be distinct Lie algebras such that the intersection $\bigcap_{i=1}^{m} U_{\fv_i}$ is not empty. Then there exist distinct indices $1\le i_1,...,i_m \le r$ such that $\fv_k\in\calr_{i_k}$ for every $1\le k\le m$. In particular, $m\le r$. 
\end{lem}

We are especially interested in restricting the covering in \eqref{eq: cover defn} to an orbit $A\pi(g)$ for some fixed $g\in G$.
Thus, for any $\fu\in\calr$ we denote 
\begin{equation*}
	U_{\fu}^{Ag}:=(U_{\fu}g^{-1})\cap A=\left\{a\in A\::\:\Ad(ag)\fu\subset\operatorname{span}\left(\Ad(ag)\fg_\ZZ\cap W_0\right)\right\}.
\end{equation*}

Let $X(T)$ be the group of \emph{$\RR$-characters} of $T$. Characters are written additively and are identified with their derivatives, that is, we think of a character as a linear functional on $\mathrm{Lie}(T)$, and $\lambda(t) = \lambda(\mathtt{t})$ for $t=\exp(\mathtt{t})\in T$ and $\lambda\in X(T)$. Fixing a norm $\|\cdot\|$ on $\operatorname{Lie}(T)$, we further use this identification of $T$ with $\operatorname{Lie}(T)$ to 
denote $\|t\|=\|\mathtt{t}\|$ for $t=\exp(\mathtt{t})\in T$.

Let $\varrho_1,\dots,\varrho_r$ and $V_1,\dots,V_r$ be as in \S \ref{sec: compactness criterion}. 
For each $1\le j\le r$, let $\Phi_{j}$ be the set of $\RR$-weights for $\varrho_j$, i.e., the set of characters $\lambda$ of $T$ for which there exists a non-zero vector $v\in V$ such that \[
\varrho_j\left(t\right)v=e^{\lambda\left(t\right)}v \]
for all $t\in T$. 

Recall the definition of $\varphi_{\lambda}$, $\lambda\in\Phi_{j}$, $1\le j\le r$ from \S \ref{sec: bounded Bruhat}. Equip $V_1,\dots,V_r$ with norms such that for each $1\le j\le r$ and any $v\in V_j$, \[
\norm{v}=\max_{\lambda\in\Phi_{j}}\norm{\varphi_{\lambda}(v)}. \]


\begin{prop}\label{prop: cover locally finite and char lin indp}
	There is a finite set $\Psi\subset X(T)$ which satisfies the following. Given $\fu\in\calr$, there exists a finite set of non-zero  linear functionals  $\Psi_{\fu}\subset \Psi$ and scalars $d_{\fu, \lambda}$, $\lambda\in\Psi$, such that 
	\begin{align}\label{eq: U tilde defn}
	    U_{\fu}^{Ag}&\subset 
	    {U}_{\fu, 0}^{Ag}:= \left\{a\in A\::\:\forall \lambda\in\Psi_{\fu},\:\lambda(a)\ge d_{\fu, \lambda}\right\}
	\end{align}
	\begin{enumerate}
		\item\label{point: the enlargement is locally finite}
		The collection $\{{U}^{Ag}_{\fu, 0}:\fu\in\calr\}$ is locally finite, that is, for any compact set $K$ there are only finitely many elements $\fu\in\calr$ such that ${U}^{Ag}_{\fu, 0}$ intersects $K$.
		\item\label{point: linearly independent characters} Assume  $\fv_1,\dots,\fv_m\in\calr$ are different unipotent radicals. If the intersection $\bigcap_{i=1}^{m} U^{Ag}_{\fv_i}$ is non-empty, then there exist linearly independent characters 
		$\lambda_1,\dots,\lambda_m\in X(T)$ such that for any $1\le i\le m$, $\lambda_i\in\Psi_{\fv_i}$. 
	\end{enumerate}
\end{prop}
\begin{remark}
Note that the characters in Assertion \eqref{point: linearly independent characters} are linearly independent only as characters over $T$. 
Their restrictions to $A$ may be linearly dependent. This is in fact the reason that the Zariski open set is used in Theorem \ref{thm: alternative intersection} and not the entire Grassmannian. We will only allow groups $A$ for which the restrictions remain linearly independent.
\end{remark}

\begin{proof}
	Let $\fu\in\calr_j$ for some $1\le j\le r$. Let 
	\begin{equation}\label{eq: defn of Psi_fu}
	    \Psi_\fu:=\left\{\lambda:-\lambda\in\Phi_j,\: \varphi_{-\lambda}(\varrho_j(g)\bp_{\fu})\neq0\right\},
	\end{equation}
	and $d_{\fu, \lambda}:=\log\norm{\varphi_{-\lambda}(\varrho_j(g)\bp_{\fu})}$ for all $\lambda\in\Psi_\fu$.  Note that $\Psi_\fu\subseteq\Psi:= -\bigcup_{j=1}^r \Phi_j$ and $\Psi$ is a finite set. 
	In addition, this implies that 
	\begin{align}\label{eq: alternative descriptoin of U^Ag_fu 0}
	    {U}^{Ag}_{\fu, 0} = \{a\in A: \varrho_j(ag)\bp_\fu\in B_j\}, 
	\end{align}
	where $B_j = \{v\in V_j:\|\varphi_\lambda(v)\|<1, \forall \lambda\in \Phi_j\}$.
	The choice of $W_0$ now implies \eqref{eq: U tilde defn}.

	Assertion \eqref{point: the enlargement is locally finite} follows from the discreteness of the set \[\{\varrho_j(g)\bp_{\fu}: \fu\in\calr_j,\:1\le j\le r\},\] as follows.
	Assume $K\subset A$ is a compact set which intersects $U^{Ag}_{\fv_i,0}$ for some infinite sequence $\fv_1,\ldots,\fv_i,\ldots\in\calr$. 
	After switching to a subsequence, we may assume that  $\fv_1,\ldots,\fv_i,\ldots\in\calr_{j_0}$ for some $1\le j_0\le r$.
    By the definition of $U^{Ag}_{\fv_i,0}$, we see that the $\varphi_{\lambda}(\bp_{\fv_i})$ are uniformly bounded for all $i\ge 0$ and $\lambda \in \Phi_j$. 
	Consequently, the sequence $(\bp_{\fv_i})_{i=0}^\infty$ lies in a compact subset of $\bigwedge^{d_j}\fg$. This is not possible because $\bp_{\fv_i}\in \bigwedge^{d_j}\fg_{\ZZ}$ and there are only finitely many integer points in a compact set.
	
	In order to show Assertion \eqref{point: linearly independent characters}, assume  $a\in\bigcap_{i=1}^{m} U^{Ag}_{\fv_i}$. Then, \[  \Ad(ag)\fv_{1},\dots,\Ad(ag)\fv_{m}\subset \spa(\Ad(g)\fg_\ZZ\cap W_0).\]
	By Proposition \ref{prop: W defn}, the space $\spa(\Ad(ag)\fv_{1},\dots,\Ad(ag)\fv_{m})$ is unipotent. Then, by Proposition \ref{prop: at most r intersections} we may assume that for some $1\le j_1<\cdots<j_m\le r$ we have that each $1\le i\le m$ satisfies $\mathfrak{v}_i\in\calr_{j_i}$.    
	Now, Corollary \ref{cor: one conjugating elemnt for intersections} implies that there exists $h\in G$ such that \[
	\fv_i=\Ad(h)\fu_{j_i}\]
	for all $1\le i\le m$. 
	In particular, for any $1\le i\le m$
	\begin{equation}\label{eq: frakv_i conj by h}
		\bp_{\fv_i}=\varrho_{j_i}(h)\bp_{\fu_{j_i}}. 
	\end{equation}
	
	By Observation \ref{obs: fundamental representations} $\varrho_1,\dots,\varrho_r$ are $\RR$-highest weight representations. For any $1\le j\le r$ the highest weight of $\varrho_j$ is denoted by $\chi_j$, and $\bp_{\fu_j}$ is a highest weight vector. 
	Then, by Corollary \ref{cor: weyl elem norm bound} and \eqref{eq: frakv_i conj by h}, there exists $w\in W(\Phi_{\RR})$ such that for all $1\le i\le m$, \[
	w(\chi_{j_i})\in\Psi_{\fv_i}.\] It is enough to show that $w(\chi_{1}),\dots,w(\chi_{r})$ are linearly independent. 
	
	It follows from Eq. \eqref{eq: fundamental weights} that  $\chi_1,\dots,\chi_r$ are linearly independent. Since the Weyl group acts on $X(T)$ by isometries, $w(\chi_{1}),\dots,w(\chi_{r})$ are also linearly independent.  
\end{proof}

For an unbounded monotonically nondecreasing $f:[0,\infty)\to [0,\infty)$ and $\fu\in \calr$ denote
\begin{align}
    {U}_{\fu, f}^{Ag}:= \left\{a\in A\::\:\forall \lambda\in\Psi_{\fu},\:\lambda(a)\ge d_{\fu, \lambda}+f(\norm{a})\right\}  \subseteq {U}_{\fu, 0}^{Ag},
\end{align}
where $d_{\fu, \lambda}$ are as in Proposition \ref{prop: cover locally finite and char lin indp}.
\begin{prop}\label{prop: cover can be f thinner}
Let $A\pi(g_0)$ be a divergent trajectory. 
Then, there exists a function $f$ which (depends on $g_0$ and) satisfies
\[\bigcup_{\fu \in \calr} \left({U}_{\fu, f}^{Ag}\cap {U}_{\fu}^{Ag}\right) = \bigcup_{\fu \in \calr} {U}_{\fu}^{Ag}.\]
\end{prop}
\begin{proof}
    Since $A\pi(g)$ diverges, by Proposition \ref{prop: compactness criterion} and Proposition \ref{prop: bound on vectors}, 
    \[
	\min\{\norm{\varrho_j(ag)\bp_{\fu}}: a \in {U}_{\fu}^{Ag}, \fu\in \calr\}\xrightarrow{\norm{a}\to \infty} 0.\]
    Hence, there exists a non-decreasing function $f:[0,\infty)\rightarrow [0,\infty)$ such that 
     \[
	\min\{\norm{\varrho_j(ag)\bp_{\fu}}: a \in {U}_{\fu}^{Ag}, \fu \in \calr\} <
	e^{-f(\norm a)},\]
	whenever $a\in \bigcup _{\fu\in \calr} {U}_{\fu}^{Ag}$.
	We get the following alternative description of ${U}_{\fu, f}^{Ag}$, which is similar to the description of $U^{Ag}_{\fu, 0}$ in  \eqref{eq: alternative descriptoin of U^Ag_fu 0},
	\begin{align}\label{eq: alternative descriptoin of U^Ag_fu f}
	    {U}^{Ag}_{\fu, 0} = \{a\in A: \varrho_j(ag)\bp_\fu\in e^{-f(\|a\|)}B_j\}, 
	\end{align}
    with $B_j$ as in the proof of Proposition \ref{prop: cover locally finite and char lin indp}. 
	
	Let $a\in \bigcup _{\fu\in \calr} {U}_{\fu}^{Ag}$. Then, there exists $\fu\in \calr$ such that $a\in {U}_{\fu}^{Ag}$. By the definition of $f$ we have \[\norm{\varrho_j(ag)\bp_{\fu}} < e^{-f(\norm a)}.\]
	Now, \eqref{eq: alternative descriptoin of U^Ag_fu f} implies that $a\in {U}_{\fu, f}^{Ag}$, which implies the claim.
\end{proof}

\begin{observation}\label{obs: f Lipschitz}
    Note that in the statement of Proposition \ref{prop: cover can be f thinner}  the function $f$ can be chosen to be $\varepsilon$-Lipschitz for every $\varepsilon>0$. 
\end{observation}

The next observation follows from \eqref{eq: alternative descriptoin of U^Ag_fu f}.
\begin{observation}\label{observation: limit of elements in U doubletilde}
Assume that $f:[0,\infty)\to [0,\infty)$ is monotone nondecreasing and unbounded.
    For any $g\in G$, $\fu\in\calr_j$, $1\le j\le r$, and $\{a_k\}\subset {U}^{Ag}_{\fu, f}$ such that $a_k\rightarrow\infty$ as $k\rightarrow\infty$, we have \[\varrho_j(a_kg)\bp_\fu\underset{k\rightarrow\infty}{\longrightarrow}0. \]
\end{observation}

For any compact set $K\subset G$ let \[
U_{\fu}^{K,Ag}:=\left\{(k,a)\in K\times A\::\: kag\in U_\fu\right\}\]

\begin{cor}\label{cor: compact trajectory perturbation}
    For any compact set $K\subset G$ and any $\fu\in\calr$, there exists a finite set $\Psi_\fu\subset X(T)$ and $d_{\fu, \lambda}'$, $\lambda\in\Psi$ such that $U_\fu^{K,Ag}\subset K\times \widetilde{U}_\fu^{K,Ag}$, where \[
    \widetilde{U}_\fu^{K,Ag}:=\left\{a\in A\::\:\forall\lambda\in\Psi_\fu,\:\lambda(a)>d_{\fu, \lambda}'\right\}.\]
    Moreover, the collection $\{\widetilde{U}^{K,Ag}_\fu:\fu\in\calr\}$ is locally finite. 
\end{cor}

\begin{proof}
    We choose $\Psi_\fu$ and $d_{\fu, \lambda}$ for $\lambda\in\Psi_\fu$ as in the proof of Proposition \ref{prop: cover locally finite and char lin indp} (Eq. \eqref{eq: defn of Psi_fu}). 
    Since $K$ is compact, there exists $c>0$ such that for any $v\in V_j$, $j=1,\dots,r$,\[
    \norm{v}\le c\quad\Longrightarrow\quad\norm{\varrho_j(k)v}\le 1. \]
    Let $d_{\fu, \lambda}' := d_{\fu, \lambda} + \log c$.
    The claim follows via similar computations to the ones in the proof of Proposition \ref{prop: cover locally finite and char lin indp}. 
\end{proof}

\section{Covering theorem} 
\label{sub: covering_theorem}
The main result of this section is Theorem \ref{thm: covering}, which builds a general machinery to disprove the existence of coverings of $n$-dimensional manifolds with covering number at most $n$. 
The machinery requires the approximation of the covering sets and their intersections by open sets for which a certain cohomology group is trivial. 

A consequence of Theorem \ref{thm: covering} is 
Corollary \ref{cor: R_n convex cover}, which approximates open sets with their convex hulls and is used in order to prove Theorem \ref{thm: alternative intersection}. A special case of Corollary \ref{cor: R_n convex cover} is \cite[Theorem 1.4]{S}.
Another consequence of Theorem \ref{thm: covering} is Theorem \ref{thm: covering_bordered}, which approximates open sets by a family of sets that are contractible and almost convex, and we call \emph{bordered}. This result will be used in order to prove Theorem \ref{thm: divergence classification}. 
The techniques used in order to prove Theorem \ref{thm: covering} are enhancements of the techniques used in order to prove \cite[Theorem 1.4]{S}, which are themselves a simplification of ideas that appear in \cite{Mc}.
\begin{de}\label{de:forms}
Let $U\subset \RR^n$ be an open set. Let $\Omega^k(U)$ denote the set of all differential $k$-forms on $U$, and let 
\[\Omega_{\operatorname{bs}}^k(U) = \bigcup_{B\subset U}\ker(\Omega^k(U)\to \Omega^k(U\setminus B)),\]
where the union is taken over all relatively closed bounded sets $B\subset U$, denote
the set of differential $k$-forms on $U$ with bounded support. 
Let $d:\Omega_{\operatorname{bs}}^k(U)\to \Omega_{\operatorname{bs}}^{k+1}(U)$ denote the standard differential, and let $H_{\operatorname{bs}}^k(U)$ be the cohomology of the complex $\Omega_{\operatorname{bs}}^k(U)$. 
An open set $U$ is called \emph{$k$-trivial} if $H_{\operatorname{bs}}^k(U)=0$. 
\end{de}

\begin{de}
An open cover is {\em locally finite} if every compact set intersects only finitely many elements of the cover.
\end{de}
Recall that a continuous map $\tau: X\to Y$ between topological spaces is \emph{proper} if the inverse image of every compact set is compact. 
If $\tau$ is proper and $X, Y$ are orientable manifolds of equal dimension, then the \emph{degree} of $\tau$ is, roughly speaking, the number of inverse images of some (and any) point $y\in Y$ counted with the correct signs (see \cite[\S 3]{gui}). Although \cite[\S 3]{gui} presents the theory of degree only for maps between compact manifolds, the theory can be applied to proper maps with similar proofs.
\begin{thm}\label{thm: covering}
Let $M$ be an orientable boundaryless manifold with $\dim M = n$ and let $\tau: M\to \RR^n$ be a proper smooth map of non-zero degree.
Let $\fU$ be an open covering of $M$. 
Let \[\{E(U_1,\dots,U_k): k\le n, U_1,\dots,U_k\in \fU\}\]
be a collection of open subsets of $\RR^n$ that satisfies the following properties:
\begin{enumerate}[label=\emph{(\arabic*)}, ref=\arabic*]
	\item \label{cond: monotone} $E(U_1,...,U_{k})\subset E(U_1,...,\widehat{U_{j}},...,U_{k})$ for every $1\le j\le k$.
	\item \label{cond: containment} $U_1\cap ...\cap U_{k} \subset \tau^{-1}(E(U_1,...,U_{k}))$, which implies that $\{E(U): U\in \fU\}$ is a covering of $\RR^n$. 
	\item\label{cond: triviality} The set $E(U_1,...,U_k)$ is $(n-k+1)$-trivial. 
	\item\label{cond: locally finiteness} The covering $\{E(U): U\in \fU\}$ of $\RR^n$ is locally finite.
\end{enumerate}
Then there exist $U_1,...,U_{n+1}\in \fU$ with nontrivial intersection. Moreover, for every $x\in \RR^n$ there exist such $U_1,...,U_{n+1}$ with $x\in \overline{E(U_1)}$.
\end{thm}
Before proving Theorem \ref{thm: covering} we present several consequences that are used in order to prove Theorems \ref{thm: divergence classification} and \ref{thm: alternative intersection}.
Corollary \ref{cor: R_n convex cover} is a special case of Theorem \ref{thm: covering}, when (heuristically) setting $E(U_1,...,U_k) = \conv(\tau(U_1\cap\dots\cap U_k))$. To apply Theorem \ref{thm: covering} we need to understand when a convex set is $k$-trivial.
\begin{de}
	The {\em invariance dimension} of a convex open set $U\subset \RR^n$ is the dimension of its stabilizer where $\RR^n$ acts on its subsets by translations, that is, \[\invdim U:=\dim \stab_{\RR^n}(U).\]
	By convention, $\invdim \emptyset:=-\infty$. 
\end{de}
The next claim follows from \cite[Theorem 3.3]{S}. 
\begin{claim}
\label{claim: invdim = k-trivial}
Let $U\subset \RR^n$ be a convex open set and $1\le k\le n$. The set $U$ is not $k$-trivial if and only if $k=\invdim U$ and $U/\stab_{\RR^n}(U) \subset \RR^n/\stab_{\RR^n}(U)$ is bounded.
In particular, $U$ is $k$-trivial for every $k\neq \invdim U$.
\end{claim}
\begin{cor}\label{cor: R_n convex cover}
	Let $M$ be an orientable boundaryless manifold with $\dim M = n$ and $\tau: M\to \RR^n$ be a proper smooth map of non-zero degree.
	Let $\fU$ be an open covering of $M$. 
	Assume that the following assertions hold. 
	\begin{enumerate}
		\item The cover \[\{\conv (\tau (U)): U\in \fU\}\] is locally finite.
		\item For every $k\le n$ and $U_1, U_2, ..., U_k\in \fU$ with nontrivial intersections one has \[\invdim \conv (\tau(U_1\cap U_2\cap ...\cap U_k)) \neq n-k+1.\]
	\end{enumerate} Then there are $n+1$ sets in $\fU$ with nontrivial intersection. 
\end{cor}
\begin{proof}
Using Claim \ref{claim: invdim = k-trivial}, the result follows directly from Theorem \ref{thm: covering} with $E(U_1,...,U_k) = \conv(\tau(U_1\cap ...\cap U_k))+B(1)$. We take the Minkovski sum with the unit ball $B(1)$ in the definition of $E(U_1,..., U_k)$ to ensure that the convex set is open.
\end{proof}

\subsection{Proof of Theorem \ref{thm: covering}} 
\label{sub: proof_of_theorem_thm: covering}
Let $M$ be an orientable boundaryless manifold with $\dim M = n$ let and $\tau: M\to \RR^n$ be a proper smooth map of non-zero degree.
Let $\fU$ and $E$ be as in Theorem \ref{thm: covering}.

For every finite sequence $J\subset \fU$ denote the intersection of the sets in $J$ by $U_J := \bigcap_{U\in J}U$. All sequences $J\subset \fU$ are assumed to be without repetition.
We say that $X\subset M$ is \emph{bounded} if its image $\tau(X)$ is bounded. 
With this notion, we can define boundedly supported $q$-forms $\Omega_{\rm bs}^{q}(U)$ for open subsets $U\subset M$ and the boundedly supported cohomology $H_{\rm bs}^{q}(U)$ similarly to the way these notions are defined on subsets of $\RR^n$ in Definition \ref{de:forms}. 
As mentioned before, by \cite[Theorem 3.3]{S} we have that $H_{\rm bs}^n(\RR^n) \cong \RR$. 

Note that for every set $W\subset\RR^n$ there is a pullback map $\tau^*:\Omega_{\rm bs}^q(W)\to \Omega_{\rm bs}^q(\tau^{-1}(W))$. 
In particular, for every finite sequence $J\subset \fU$ of size at most $n+1$ we have a map $\tau^*:\Omega^q_{\rm bs}(E(J))\to \Omega^q_{\rm bs}(U_J)$. Since $\tau^{*}$ commutes with $d$, it induces a map $\tau^*: H_{\rm bs}^q(W)\to H_{\rm bs}^q(\tau^{-1}(W))$

\begin{lem}\label{lem: tau 121}
The map $\tau^*: H^{n}_{\rm bs}(\RR^n) \to H^{n}_{\rm bs}(M)$ is one to one. 
\end{lem} 
\begin{proof}
Since both $\RR^n$ and $M$ are $n$-dimensional the integration map satisfies 
\[\xymatrix{H^{n}_{\rm bs}(\RR^n) \ar[r]^{\tau^*}\ar[d]^{\int}& H^{n}_{\rm bs}(M)\ar[d]^{\int}\\\RR\ar[r]^{\deg \tau}&\RR}\]
which implies the desired result, using $H^n_{\rm bs}(\RR^n)\cong \RR$. 
\end{proof}

We will construct two complexes: 
\[\cA^{p,q} := \bigoplus _{J\subset \fU, \#J=p+1}\Omega_{\rm bs}^{q}(U_J),\quad  \forall p, q\ge 0\]
and 
\[\cB^{p,q} := \bigoplus _{J\subset \fU, \#J=p+1}\Omega_{\rm bs}^{q}(E(J)),\quad \forall p, q\ge 0,\]
with the standard derivations
\begin{align*}
    \delta:\cA^{p,q}&\to \cA^{p+1, q},& \delta:\cB^{p,q}&\to \cB^{p+1, q},\\ d:\cA^{p,q}&\to \cA^{p, q+1},& d:\cB^{p,q}&\to \cB^{p, q+1},
\end{align*}
as in \cite[\S 10]{BT2} or \cite[\S 3.4]{S}. The maps $\tau^*:\Omega_{\rm bs}^{q}(E(J))\to \Omega_{\rm bs}^{q}(U_J)$ for $J\subset \fU$ of size $p+1$ can be summed to a single map $\tau^*:\cB^{p, q}\mapsto \cA^{p, q}$. Note that the derivations $d$ and $\delta$ commute with $\tau^*$ and $d$ commutes with $\delta$, i.e., $\delta d - d\delta = 0$.

There is a mild subtlety here. In the statement of Theorem \ref{thm: covering}, the sets $E(J)$ are assumed to only be defined for sequences $J$ of size at most $n$, while in the definition of $\cB^{*,*}$ we use $E(J)$ for larger sequences $J$ as well. Thus, we define 
\[E(J) := \bigcap_{J'\subset J,\ \#J'=n} E(J'), ~~\text {whenever }\#J>n.\] 
The above preserves Conditions (\ref{cond: monotone}), (\ref{cond: containment}) of Theorem \ref{thm: covering} regarding $E$, but not (\ref{cond: triviality}).

Denote by $\tot(\cA)^*$ and $\tot(\cB)^*$ the total complexes of $\cA^{*,*}$ and $\cB^{*,*}$, respectively. 
That is, $\tot(\cA)^k$ is given by $\tot(\cA)^k = \bigoplus_{p+q=k}\cA^{p,q}$, 
and the differential $D:\tot(\cA)^k\to \tot(\cA)^{k+1}$ is defined by 
\[D|_{\cA^{p,q}} = d+(-1)^q\delta.\]
The total complex $\tot(\cB)^*$ is defined similarly. 
Let the maps $i_{\cA}:\Omega^*_{\rm bs}(M)\to \tot(\cA)^*$ and $i_{\cB}:\Omega^*_{\rm bs}(\RR^n)\to \tot(\cB)^*$ be the direct sum of the restrictions to $\cA^{0,*}$ and $\cB^{0,*}$, respectively. 
Both maps commute with $\tau^*$, that is 
\[\xymatrix{
\Omega^*_{\rm bs}(\RR^n)\ar[r]^{i_\cB}\ar[d]^{\tau^*}& \tot(\cB)^*\ar[d]^{\tau^*}\\
\Omega^*_{\rm bs}(M)\ar[r]^{i_\cA}& \tot(\cA)^*}\]
The maps $i_\cA$ satisfies that for every $\omega \in \Omega^k(M)$ we have 
\begin{align}\label{eq: iA properties}
    d i_{\cA}(\omega) = i_{\cA}(d\omega)\quad \text{and}\quad \delta i_{\cA}(\omega)=0,
\end{align}
and similarly $i_\cB$ satisfies that for every $\omega \in \Omega^k(
\RR^n)$ we have 
\begin{align}\label{eq: iB properties}
    d i_{\cB}(\omega) = i_{\cB}(d\omega)\quad \text{and}\quad \delta i_{\cB}(\omega) = 0.
\end{align}

The next lemma is from \cite[Proposition 8.5]{BT2} or \cite[Theorem 3.7]{S}. It states that $\delta$-derivation on $\cA^{*,*}$ gives the complex $\Omega^*_{\rm bs}(M)$. We need the following formulation of it.
\begin{lem}\label{lem: import i is quaziequivalence}
The map $i_\cA:\Omega^*_{\rm bs}(M)\to \tot(\cA)^*$ induces an isomorphism on the cohomologies. 
\end{lem}

\begin{observation}[Consequence of Condition (\ref{cond: triviality})]\label{obs: d exact}
Condition (\ref{cond: triviality}) states that $d$ is exact on $\cB^{p,q}$ whenever $p+q = n$ and $p<n$.
\end{observation}

For every $\omega\in \cB^{p,q}$ denote by 
$\supp \omega$ the set of sequences $J\subset \fU$ such that the $J$-component of $\omega$ is nontrivial. 
Applying this definition to Observation \ref{obs: d exact} we get the following.

\begin{observation}
[Behavior of the support]\label{obs: implication of support}
Let $p,q$ be such that $p+q = n$, $p<n$, and $\omega\in \cB^{p,q}$. If $d\omega = 0$, then for some $\omega'\in \cB^{p,q-1}$ we have $d\omega' = \omega$ and $\supp \omega' = \supp \omega$. 
In addition, by the definition of $\delta$, for every $\omega\in \cB^{p,q-1}$ and $J\in \supp \delta\omega$ there is $J'\in \supp \omega$ with $J'\subset J$.
\end{observation}

\begin{proof}[Proof of Theorem \emph{\ref{thm: covering}}] 

Pick an element $\omega\in \Omega^n_{\rm bs}(\RR^n)$ with nontrivial integral and set $\omega_0 = i_{\cB}(\omega)\in \cB^{0, n}$. Since $d\omega = 0$ we use Eq. \eqref{eq: iB properties} to deduce that  $d\omega_0 =\delta\omega_0 = 0$.  

We construct $\omega_{k}\in \cB^{k, n-k}$ recursively so that  $d\omega_k = \delta\omega_k = 0$ and the image of $\omega_k$ in $H^n(\tot(\cB))$, denoted $[\omega_k]$, coincides with $[\omega_0]$. Suppose we have constructed $\omega_{k-1}\in \cB^{k-1, n-k+1}$. Then, by Observation \ref{obs: implication of support} there exists $\psi_k\in \cB^{k-1, n-k}$ such that $d\psi_k = \omega_{k-1}$. Choose $\omega_k = -(-1)^{n-k}\delta\psi_k$. Then, $\omega_{k-1}-\omega_k = D\psi_k$ is a boundary in $\tot(\cB)$ and hence $[\omega_{k-1}] = [\omega_k]$.
Note that
\begin{align*}
    d\omega_k &= -(-1)^{n-k}d\delta\psi_k = -(-1)^{n-k}\delta d\psi_k = -(-1)^{n-k}\delta\omega_{k-1} = 0\\
    \delta\omega_k &= -(-1)^{n-k}\delta^2\psi_k = 0,
\end{align*}
as desired. 
We get that $[\omega_0] = [\omega_n]$ in $H^n(\tot(\cB))$. Hence, we also have $\tau^*[\omega_0] = \tau^*[\omega_n]$.
By Lemma \ref{lem: tau 121}, $[\omega]\neq 0$ implies that $0\neq \tau^*[\omega]$. By Lemma \ref{lem: import i is quaziequivalence}, $0\neq i_{\cA}^*\tau^*[\omega]$. Since
\begin{align*}
0\neq i_\cA^*\tau^*[\omega] = \tau^*i_\cB^*[\omega] = \tau^*[\omega_0] = \tau^* [\omega_n],
\end{align*}
we get that $\tau^*\omega_n\neq 0$. Since $\tau^*\omega_n\in \cA^{n,0}$ it follows that there is $J\subset \fU_n$ of size $n+1$ with nonempty $U_J$. 
Moreover, this $J$ must be in $\supp \omega_n$.

Let $x\in \RR^n$. Since $\{E(U): U\in \fU\}$ is locally finite, there is a small neighborhood $x\in V\subset \RR^n$ such that $E(U)\cap V\neq \emptyset$ if and only if $x\in \overline{E(U)}$ for every $U\in \fU$. Additionally, there are finitely many $U_1,...,U_r\in \fU$ such that $x\in \overline {E(U_i)}$. 
Choose $\omega\in \Omega^n(\RR^n)$ which is supported on $V$. 

By our construction, $\supp \omega_0 \subseteq \{(U_1,),...,(U_r,)\}$. Here $(U_i,)$ is a sequence of size $1$ of indices in $\fU$. Since $\psi_1$ was chosen using  Observation \ref{obs: implication of support}, it has the same support as $\omega_0$. 
Looking at $\supp\omega_1$, we see that every $J\in \supp\omega_0$ must contain an element in $\supp\psi_1 = \supp\omega_0$, i.e., one of the desired $U_i$-s.

Continuing in this fashion, we conclude that $\supp\psi_k = \supp \omega_{k-1}$ and that every $J\in \supp\omega_k$ contains at least one of the $U_i$-s. 
That is, for some $J\in \supp \omega_n$ we have $U_J\neq \emptyset$, and this $J$ contains one of the desired open sets.
\end{proof}





\section{Bordered sets}
\label{sec: bordered sets}
In \S \ref{sec: the cover} we saw that a divergent trajectory gives rise to a cover of $A$; in view of Proposition \ref{prop: cover can be f thinner}, the sets in this cover form a peculiar family of shapes, namely $U^{Ag}_{\fu, f}$. In this section, we analyze these sets and derive sufficient conditions for their contractibility.

Fix a norm $\|\cdot \|$ on $\RR^n$ through out this section. It may not be the euclidean norm, even though we use the standard scalar product on $\RR^n$.
\begin{de}\label{defn: bordered set}
Let $\Phi\subset(\RR^n)^*$ be a finite set of linear functionals, and let $f:[0,\infty)\to \RR$ 
be an increasing divergent function. 
An open set of the form 
\[U=\{  x\in\RR^n:\varphi(  x) > C_\varphi +  f(\|  x\|)\text{ for all }\varphi\in \Phi'\}\]
for some $\emptyset \neq \Phi' \subset \Phi$ and $(C_\varphi)_{\varphi\in \Phi'}\subset \RR$ is a {\em $(\Phi, f)$-bordered set}.
The set $\Phi'$ is called the {\em functional set} of $U$. If $\Phi$ and $f$ are clear from the context we will omit them and simply call the set bordered.
Note that a finite intersection of bordered sets is itself a bordered set. 

A set of vectors $S$ in a real vector space $V$ is said to be {\em positively nontrivial} if every nontrivial non-negative combination of them is nontrivial, equivalently, if  \[
0\notin\operatorname{conv}(S). \]
\end{de}

The following result lists some properties of bordered sets.
\begin{thm}\label{thm: bordered properties}
For every $\Phi$ there exists $\varepsilon>0$ such that for every $\varepsilon$-Lipschitz increasing divergent $f:[0,\infty)\to [0,\infty)$ the following assertions hold:
\begin{enumerate}[label=\emph{(\arabic*)}, ref=\arabic*]
	\item \label{prop: contractible} Every nonempty bordered set is contractible. 
	\item A bordered set is bounded if and only if its functional set is not positively nontrivial. 
	\item \label{prop: no boundedly supported rcohomologies} 
	 If a bordered set $U$ is unbounded, then there are arbitrarily large bounded subsets $U_0\subset U$ such that $U\setminus U_0$ is contractible.
\end{enumerate}
\end{thm}
\begin{remark}
The constant $\varepsilon$ depends also on the norm $\|\cdot \|$ on $\RR^n$. 
\end{remark}

We postpone the proof of Theorem \ref{thm: bordered properties} to later on in this section and list next a few of its corollaries.
Note that Property \ref{prop: no boundedly supported rcohomologies} is not shared by convex sets. In particular, a cylinder does not satisfy it. Formally, Property \ref{prop: no boundedly supported rcohomologies} implies the following.
\begin{cor}\label{cor: bordered are trivial}
Every bordered set is $k$-trivial for every $k > 0$. 
\end{cor}
Combining Corollary \ref{cor: bordered are trivial} with Theorem \ref{thm: covering}, and recalling that a finite intersection of bordered sets is bordered, we obtain 
\begin{thm}\label{thm: covering_bordered}
Let $\Phi,f$ be as in Theorem \emph{\ref{thm: bordered properties}}. Let $\fU$ be an open covering of $\RR^n$. Assume that for every $U\in \fU$ there is given a bordered set $E(U)$ with $U\subset E(U)\subsetneq\RR^n$.
Assume furthermore that the covering $\{E(U): U\in \fU\}$ is locally finite.
Then, there exist $U_1,...,U_{n+1}\in \fU$ with a nontrivial intersection. Moreover, for every $x\in \RR^n$ there exist such $U_1,...,U_{n+1}$ satisfying in addition $x\in \overline{E(U_1)}$.
\end{thm}

\begin{proof}[Proof of Theorem \emph{\ref{thm: bordered properties}}]
Fix $\varepsilon>0$ to be determined later and an $\varepsilon$-Lipschitz and divergent function $f:[0,\infty)\to [0,\infty)$. We need upper bounds on $\varepsilon$ throughout the proof, and choose $\varepsilon$ to be the minimum of these upper bounds.
Without loss of generality restrict $\Phi$ to be the functional set of $U$, and assume 
\[U = \{  x\in\RR^n:\varphi(  x) > C_\varphi +  f(\|  x\|)\text{ for } \varphi \in \Phi\}.\]

We distinguish between two cases: 
\begin{enumerate}[label=(\alph*)]
	\item \label{case: compact} $\Phi$ is not positively nontrivial.
	\item \label{case: non compact} $\Phi$ is positively nontrivial. 
\end{enumerate}

\noindent{\bf Case \ref{case: compact}:}
Define
\[\varrho, \varrho_0:\RR^n \to \RR;~~\varrho_0(  x) := \min_{\varphi\in \Phi}\varphi(  x)-C_\varphi\text{ and }\varrho(  x) := \varrho_0(  x) - f(\|  x\|).\]
Note that $U = \{  x\in \RR^n:\varrho(  x) > 0\}$ and $\varrho_0 \ge \varrho$. 
We will construct a contraction map of $\RR^n$ and show that it expands $\varrho$, thus it contracts $U$. 

By assumption, there is a combination $\sum_{\varphi\in \Phi} \varphi\alpha_\varphi = 0$. Therefore, for every $x$ one of the functions $\varphi_i$ is nonpositive at $x$, and hence $\varrho_0$ is bounded from above. Since $\varrho_0$ is piecewise linear with finitely many possible slopes, it attains a maximum $M$.
Denote by $V = \{  x\in \RR^n:\varrho_0(  x) = M\}$, which is a closed convex polytope. Let $A:\RR^n\to V$ be the map that assigns to each vector $  x\in \RR^n$ the unique closest vector $A(  x) \in V$. The map $A$ is continuous since $V$ is convex. 

The first part of the contraction will move every vector linearly from $  x$ to $A(  x)$. 
\begin{claim}\label{claim: trajectory up}
The map $t\mapsto \varrho(tA(  x) + (1-t)   x)$ is monotonically nondecreasing for $t\in [0,1]$.
\end{claim}
As $\varrho = \varrho_0-f$, to prove the claim
we will show that $\varrho_0$ increases faster then $  v\mapsto f(\|  v\|)$ along the trajectory $t\mapsto tA(  x) + (1-t)   x$. 
\begin{claim}\label{claim: trajectory0 up}
The map $\psi: t\mapsto \varrho_0(tA(  x) + (1-t)   x)+ \varepsilon (1-t) \|A(  x)-  x\|$ is monotonically nondecreasing for $t\in [0,1]$, provided that $\varepsilon$ is small enough.
\end{claim}
\begin{proof}
If $A(  x) =   x$, then $\psi$ is constant and there is nothing to prove. Assume now that $A(  x) \neq   x$.
Claim \ref{claim: trajectory0 up} does not involve $f$ and is invariant under translations. Assume without loss of generality that $A(  x) =   0$. Since $V$ is a polytope, there is a neighborhood $0\in W\subset\RR^n$ of $  0=A(  x)$ such that $V\cap W$ is defined only by inequalities $\varphi_i-C_i\ge M$ that become equality at $  0$, that is, $\Phi'=\{\varphi\in \Phi:-C_\varphi = M\}$. In other words, $\Phi'$ satisfies that 
\[V\cap W = \{  v\in W:\varphi(  v) \ge 0 \text{ for } \varphi\in \Phi'\}.\]

Since locally $0\in V$ is the closest point to $x$, it follows that $\bra x, v\ket \le 0$ for every $  v\in V\cap W$, and hence $-  x$ lies in the dual cone to $V\cap W$. This dual cone is generated by the vectors $  v_{\varphi}$ for $\varphi\in \Phi'$, where $  v_\varphi$ is the unique vector that satisfies 
$\bra   v_{\varphi},   y\ket = \varphi(  y)$ for all $  y\in \RR^n$. 
Hence, there is a negative combination $  x = \sum_{\varphi\in \Phi'}\beta_\varphi   v_{\varphi}$. Consequently, there is a linearly independent $\Phi''\subset \Phi'$ such that $x = \sum_{\varphi\in \Phi''}\gamma_\varphi  v_{\varphi}$ for negative $\gamma_\varphi$. 
Since $\bra x, x\ket >0$, there exists $\varphi_0\in \Phi''$ such that $\varphi_0(  x) < 0$. Lemma \ref{lem: find varepsilon} gives an upper bound on $\varepsilon>0$ that depends only on $\Phi''$ and guarantees that $\varphi_0(  x) \le -\varepsilon \|x\|$. 

The function $\psi$ is concave, takes the value $M$ at $t=1$ and satisfies 
\begin{align*}
\psi(t)
&=\varrho_0(tA(  x)+ (1-t)   x) + \varepsilon (1-t)\|A(  x)-  x\|  \\&\le \varphi_0(tA(  x) + (1-t)   x) - C_{\varphi_0} + \varepsilon (1-t)\|A(  x)-  x\| \\ & = 
(1-t)\varphi_0(  x) - C_{\varphi_0} + \varepsilon (1-t)\|  x\| \\& \le -C_{\varphi_0} = M.
\end{align*}
Thus $\psi$ is bounded by its value at $1$. Since $\psi$ is concave, it is nondecreasing for $t\in [0,1]$.
\end{proof}
\begin{lem}\label{lem: find varepsilon}
For every linearly independent subset $\Phi'\subset \Phi$ and every $  x = \sum_{\varphi\in \Phi'}\beta_\varphi v_{\varphi}$ with $\beta_\varphi \le 0$ there is $\varphi\in \Phi'$ with $\varphi(  x)\le -\varepsilon\|  x\|$, provided that $\varepsilon>0$ is small enough.
\end{lem}
\begin{proof}
For every linearly independent subset $\Phi'\subset \Phi$ we will give an upper bound on $\varepsilon$ separately such that $\varphi(  x) \le \varepsilon \|  x\|$ for some $\varphi \in \Phi'$. Assume then that $\Phi = \Phi'$ is an independent set of functionals. 
Since the inequality we want to prove is homogeneous, we may assume $\|  x\| = 1$. 
The set of possible values of $  x$ is
\[X := \left\{\sum_{\varphi\in \Phi}\beta_\varphi v_{\varphi} :\left\|\sum_{\varphi\in \Phi}\beta_\varphi v_{\varphi}\right\| = 1, \beta_\varphi\le 0\right\},\]
which is compact. 
Hence we only need to show that 
\[\min_{\varphi\in \Phi} \varphi(  x) < 0, \quad \forall x\in X\]
and the compactness of $X$ and the continuity of the function $\min_{\varphi\in \Phi} \varphi$ will guarantee a desired bound on $\varepsilon$. Since
 \[\bra   x,   x\ket = \bra   x, \sum_{\varphi\in \Phi}\beta_\varphi   v_\varphi\ket = \sum_{\varphi\in \Phi}\beta_\varphi\varphi(  x)=1,\]
and since the $(\beta_\varphi)_{\varphi\in \Phi}$ are nonpositive, it follows that one of the $\varphi(  x)$ is negative.
\end{proof}
\begin{proof}[Proof of Claim \emph{\ref{claim: trajectory up}}]
Since $f$ is $\varepsilon$-Lipschitz, $\|\cdot\|:\RR^n \to \RR$ is $1$-Lipschitz, and $t\mapsto tA(  x) + (1-t)   x$ is $\|  x-A(  x)\|$-Lipschitz, it follows that $t\mapsto f(\|tA(  x) + (1-t)   x\|)$ is $\varepsilon\|  x-A(  x)\|$-Lipschitz. By Claim \ref{claim: trajectory0 up}, $$\varrho_0(tA(  x) + (1-t)   x)+\varepsilon (1-t)\|A(  x)-  x\|$$ is nondecreasing, hence so is $$\varrho_0(tA(  x) + (1-t)   x) -f (\|tA(  x) + (1-t)   x\|)  = \varrho(tA(  x) + (1-t)   x).$$
\end{proof}
Consider the homotopy $h_0(t,  x) := tA(  x) + (1-t)   x$, and note that $\varrho(h_0(  x, t))$ is nondecreasing with respect to $t$. 
The end of the homotopy $h_0|_{t=1}$ lies in $V$. 

Next, we construct a contraction of $V$. 
Let $  u = A(0)\in V$ be the point that minimizes the distance to $0$, and define $h_1(t,   v) = t  v + (1-t)  u$. By the convexity of $\|\cdot \|$, for every fixed $  v\in V$ the function $t\mapsto\|h_1(t,   v)\|$ is decreasing, and hence the function $t\mapsto f(\|h_1(t,   v)\|)$ is decreasing. Therefore 
\[t\mapsto\varrho(h_1(t,   v)) = M-f(\|h_1(t,   v)\|)\]
is increasing. It follows that $\varrho$ increases on every trajectory of $h_1$. 
Thus, the concatenation of $h_0$ and $h_1$ is a contraction of $\RR^n$ and $\varrho$ increases along its trajectories.

Hence, $U = \{  x\in \RR^n:\varrho(  x)>0\} $ is contractible if it is nonempty. 

Since $\varrho_0$ is bounded and $\lim_{\|  x\|\to \infty}f(\|  x\|)=\infty$, it follows that $\varrho(  x)\to -\infty$ as $\|  x\|\to \infty$, and hence $U$ is bounded.

{\bf Case \ref{case: non compact}:}
In this case, the set $U$ is unbounded. Indeed, since the functional set $\Phi$ of $U$ is positively nontrivial, there is a vector $  v =   v(\Phi')$ such that $\varphi(  v)>0$ for all $\varphi\in \Phi'$. 
If $f$ is $\left(\varphi(  v)/\|  v\|\right)$-Lipschitz, then for $s>0$ we have 
$$\varphi(s  v)-f(\|s  v\|) \xrightarrow{s\to \infty}\infty.$$ It follows that $s  v\in U$ for $s$ large enough.
Moreover, if $  x\in U$ then $  x + s  v \in U$ as well for every $s>0$, and if $  x\in \RR^n$ then $x+s  v\in U$ for all $s$ large enough.
Hence, provided 
\[\varepsilon \le \min_{\varphi\in \Phi'}\frac{\varphi(  v)}{\|  v\|},\]
$U$ will not be bounded.

To show that $U$ is contractible, we will show that every compact subset $K\subset U$ can be contracted in $U$. 
Suppose $K\subset U$ is compact. It has a contraction in $\RR^n$, denoted $h: K\times [0, 1]\to \RR^n$. Since the image of $h$ is compact, there is $s>0$ such that $h(t,  x)+s  v\in U $ for all $  x \in K, t\in [0,1]$. 
Thus, to contract $K$ in $U$ we may first add to it $s  v$ and then contract $K+s  v$ using $h+s  v$. 

Since every compact subset of $U$ can be contracted in $U$, it follows that $U$ is connected and all homotopy groups $\pi_i(U)$ vanish. Since $U$ is a separable manifold, it is a $CW$-complex, and hence, by Whitehead's Theorem (see \cite[Theorem 4.5]{hatcher}), $U$ is contractible.

We will show that there exist open subsets $U_i\subset U$ for $i\in \NN$ such that 
\begin{itemize}
\item $U\setminus U_i$ is bounded
\item For every bounded set $B\subset U$ there is $i\in \NN$ such that $B\subset U\setminus U_i$.
\item $U_i$ is contractible.
\end{itemize}
Fix $\varphi_0\in \Phi$.  Since $\Phi$ is positively nontrivial, it follows that $\varphi_0 \neq 0$.
For every $C>0$ denote $U_C = U\cap \varphi_0^{-1}((C, \infty))$.
We claim that $U\setminus U_C$ is bounded.
Indeed, if $  x\in U\setminus U_C$, then $C > \varphi_0(  x) \ge C_{\varphi_0}$ and since $\varphi_0(  x)-f(\|  x\|) > C_{\varphi_0}$ it follows that $f(\|  x\|)<C-C_{\varphi_0}$, and hence $\|  x\|$ is bounded.
For every bounded subset $B\subset U$ we have that $B\subset U\setminus U_C$, where $C := \sup_{B}\varphi_0$.

To show that $U_C$ is contractible, note that $U_C$ satisfies the same properties which were used in order to show that unbounded bordered sets are contractible: It is open, if $  x\in U_C$, then so is $  x+s  v$ for every $s>0$, and for all $  x\in \RR^n$ we have $  x+s  v\in U_C$ for all $s$ large enough.
We can now show that $U_C$ is contractible by the same arguments which were used before to show that $U$ is contractible.
\end{proof}

\begin{proof}[Proof of Corollary \emph{\ref{cor: bordered are trivial}}]
If $U$ is bounded, then its boundedly supported cohomologies are equal to its standard de Rham cohomology, and the result follows since $U$ is contractible. Assume that $U$ is unbounded.
Let $\omega\in \Omega^k_{\rm bs}(U)$ with $d\omega = 0$. 
Let $V\subset U$ be a relatively closed bounded set such that $\supp \omega\subset V$ and $U\setminus V$ is contractible. 
Since $U$ is contractible, there is $\alpha\in \Omega^{k-1}$ such that $d\alpha = \omega$. 
The $(k-1)$-form $\alpha$ is not necessarily boundedly supported, hence it does not follow that $\omega$ vanishes in $H_{\rm bs}^k$. To prove that $\omega$ vanishes in $H_{\rm bs}^k$ we replace $\alpha$ with a boundedly supported alternative.
Note that $d\alpha|_{U\setminus V}=0$. 
We distinguish between two cases: $k=1$ and $k>1$. 

If $k=1$, then since $U\setminus V$ is contractible, $\alpha|_{U\setminus V}\equiv c$ is a constant. 
The $0$-form $\alpha'=\alpha-c$ is boundedly supported and satisfies $d(\alpha - c) = \omega$, as desired.

If $k>1$, then $H^{k-1}(U\setminus V) = 0$, and hence there is $\beta\in \Omega^{k-2}(U\setminus V)$ such that $d\beta = \alpha|_{U\setminus V}$. 
Pick an open bounded set $W \subset U$ that contains $V$, and consider a smooth partition of unity subordinate to the cover $U\setminus V, W$ of $U$, namely, $\varrho, 1-\varrho: U\to \RR$ such that $\supp \varrho\subset U\setminus V$ and $\supp (1-\varrho)\subset W$. 
Now, $\varrho\beta|_{U\setminus W} = \beta|_{U\setminus W}$ and we can extend $\varrho\beta$ to a form on $U$, namely, $\widetilde{\varrho\beta}\in \Omega^{k-2}(U)$. 
The form $\alpha' = \alpha - d\widetilde{\varrho\beta}$ is boundedly supported, since it coincides with $\alpha - d\beta$ on $U\setminus W$. 
It follows that $\omega = d\alpha'$, as desired.
\end{proof}

\section{Proof of Theorem \emph{\ref{thm: divergence classification}}} \label{sec: proof of divergence classification}

In order to prove Theorem \ref{thm: divergence classification}, in this section we show that the cover constructed in \S \ref{sec: the cover} (of the subgroup $A$, up to a compact set) has a finite subcover. 

Recall the definition of the sets $U^{Ag}_{\fv}$ and $U^{Ag}_{\fv,f}$, for $\fv\in \calr$, provided in \S \ref{sec: the cover}. Let $f$ satisfy the conclusion of Proposition \ref{prop: cover can be f thinner} 

\begin{prop}\label{prop: finite cover}
    Let $A\pi(g)$ be a divergent trajectory and assume $\dim A=\rank_{\QQ}G$. Then, there exist $\fv_1,\dots,\fv_ m\in\calr$ such that $A\setminus\bigcup_{i=1}^{ m} \overline{U^{Ag}_{\fv_i, f}}$ is bounded. 
\end{prop}

Assuming Proposition \ref{prop: finite cover}, Theorem \ref{thm: divergence classification} follows from Observation \ref{observation: limit of elements in U doubletilde}. Thus, the rest of this section is devoted to the proof of Proposition \ref{prop: finite cover}. 

\begin{proof}[Proof of Proposition \emph{\ref{prop: finite cover}}]
Fix a divergent trajectory $A\pi(g)$, and let $f$ be as in Proposition \ref{prop: cover can be f thinner}.
Consider the collection of sets \[\fU^{Ag}_f := \{U^{Ag}_{\fv}\cap U^{Ag}_{\fv, f}:\fv\in \calr\}.\] By Proposition \ref{prop: compactness criterion} and the definition of $U_{\fv}$ (see Eq. \eqref{eq: cover defn}) we see that, up to a compact set, the sets $U_{\fv}$ cover $A$. 
By Proposition \ref{prop: cover can be f thinner}, $\fU^{Ag}_f$ also cover $A$ up to a compact set. 
Let $U_0\subset A$ be a bounded open set such that $\bigcup_{\fv\in \calr} \left(U^{Ag}_{\fv}\cap U^{Ag}_{\fv,f}\right) \cup U_0 = A$. 
Denote \[\fU^{Ag}_{f+} := \fU^{Ag}_f \cup \{U_0\}.\] Then, $\fU^{Ag}_{f+}$ is a cover of $A$. 

We wish to apply Theorem \ref{thm: covering_bordered} to the cover $\fU^{Ag}_{f+}$ of $A$. We need to construct the sets $E(U)$ for $U\in \fU^{Ag}_{f+}$. 

For any $U_{\fv}^{Ag}\in \fU_0^{Ag}$ define $E(U_{\fv}^{Ag}) := U_{\fv, f}^{Ag}$ and take for $E(U_0)$ an arbitrary bounded bordered set containing $U_0$. 
Note that the cover $\{E(U): U\in \fU^{Ag}_{f+}\}$ is locally finite by Proposition \ref{prop: cover locally finite and char lin indp}.
To apply Theorem \ref{thm: covering_bordered} we need to make sure that the sets $E(U)$ are bordered. 
Indeed, they are constructed to be $(\Psi,f)$-bordered in Proposition \ref{prop: cover can be f thinner} for some nondecreasing monotone unbounded function $f$, which may be chosen to be $\varepsilon = \varepsilon(\Psi)$-Lipschitz as in Theorem \ref{thm: bordered properties}, by Observation \ref{obs: f Lipschitz}.
Now we are allowed to apply Theorem \ref{thm: covering_bordered} to the cover $\fU^{Ag}_{f+}$. 

Denote by $\fv_1,...,\fv_m\in \calr$ the different Lie algebras for which $U_{\fv, f}^{Ag}$ intersects $E(U_0)$. There are only finitely many of those since $\{E(U): U\in \fU^{Ag}_{f+}\}$ is locally finite. 

We claim that $A\setminus\bigcup_{i=1}^{m} \overline{U^{Ag}_{\fv_i, f}}$ is bounded, and more precisely, that 
$\overline{E(U_0)}\cup \bigcup_{i=1}^{m} \overline{U^{Ag}_{\fv_i, f}} = A$. 
Denote by $r$ the dimension of $A$, (which is assumed to be equal to $\rank_{\QQ}G$). Let $x\in A\setminus \overline{E(U_0)}$. 
By Theorem \ref{thm: covering_bordered} there exist different sets $U_1,...,U_{r+1}\in \fU^{Ag}_{f+}$ 
such that $U_1\cap ...\cap U_{r+1}\neq \emptyset$, and $x\in \overline{E(U_1)}$. 
By Lemma \ref{lem: no large intersection} at most $r$ sets of $U_1,...,U_{r+1}$ are in $\fU^{Ag}_{f}$, and hence for some $1\le i\le r$ we have $U_i=U_0$. Since $x\in \overline {U_1}$ and $x\nin \overline {E(U_0)}$, if follows that $U_1\neq U_0$, and hence $U_1 = U^{Ag}_{\fv, f}$ for some $\fv\in \calr$. 
Since $U_1\cap U_0\neq \emptyset$ it follows that $\fv = \fv_i$ for some $1\le i \le m$. The result follows.  
\end{proof}

\begin{remark}\label{rem: Barak's proof}
Similar arguments to the ones used in the proof of Proposition \ref{prop: finite cover} above yield an alternative proof to Weiss' result \cite[Corollary 2]{Weiss2} that for any $\RR$-diagonalizable subgroup $A$ which satisfies $\dim A>\rank_\QQ G$, there are no divergent trajectories in $G/\Gamma$. Explicitly, if $\ell:=\dim A>\rank_\QQ G=: r$ and $A\pi(g)$ is a divergent trajectory, then by Theorem \ref{thm: covering_bordered}, a covering $\fU^{Ag}_{f+}$  of $\operatorname{Lie}(A)$ defined as in the above, contains $\ell+1$ sets with a non-empty intersection. This is a contradiction to 
Lemma \ref{lem: no large intersection}. 
\end{remark}

\section{Avoiding compact sets} 
\label{sec: proof_of_theorem_thm: intersection}
We begin with a topological manipulation, showing that we may replace the deformation retract with the image of a compact set $Y$ which is homotopy equivalent to $G/\Gamma$. 
We then use the homotopy to push an orbit $Ax$ away from a much bigger compact set in $G/\Gamma$. 
In doing so, the set $Ax$ is pushed to a set which is no longer an actual orbit,  instead it is of the form $\{f(a)ax: a\in A\}$ for some bounded correspondence $f: A\to G$. 
Finally, we use the data of the location of the orbit to attain a cover which contradicts Corollary \ref{cor: R_n convex cover}.

\subsection{Getting a Compact Subset} 
\label{sub: getting_a_compact_subset}
In this section replace an arbitrary deformation retract of $G/\Gamma$ with a compact set which is the image of a homotopy equivalence map to $G/\Gamma$.

\begin{thm}[\cite{Raghu}, \cite{BS}]
\label{thm: there is compact deformation retract}
There exists a compact deformation retract to $G/\Gamma$. 
\end{thm}

It is hard to trace the origin of Theorem \ref{thm: there is compact deformation retract}. It follows from Raghunathan \cite{Raghu} using gradient flow on the function constructed in the theorem on page 328. It can also be derived from \cite{BS}, proving the simplicity of the space $K\backslash G / \Gamma$. Both \cite{gr2} and \cite{saper} used \cite{BS} to construct compact retracts to $K\backslash G / \Gamma$, and their methods are able to construct compact retracts to $G / \Gamma$. Here $K\subset G$ is a maximal compact subgroup.

\begin{lem}\label{lem: exists compact subset h-equiv}
For every deformation retract of $Z\subset G/\Gamma$ there exists a homotopy equivalence $f: Y\to G/\Gamma$ from a compact space $Y$,
such that ${\rm Im}(f)\subset Z$.
\end{lem}
\begin{remark}
The compact space $Y$ in Lemma \ref{lem: exists compact subset h-equiv} is independent of $Z$. Moreover, it can be taken to be any compact deformation retract of $G/\Gamma$. Also, note that the map $f: Y\to G/\Gamma$ is not the inclusion map.
\end{remark}
\begin{proof}
Let $Y\subset G/\Gamma$ satisfy the conclusion of Theorem \ref{thm: there is compact deformation retract}, i.e., be a compact deformation retract of $G/\Gamma$. Denote by $i_Y: Y\to G/\Gamma$ the inclusion map. Since $Y$ is a deformation retract, $i_Y$ is a homotopy equivalence.

Let \[\xymatrix{Z\ar@/^/[r]^{i_Z}&G/\Gamma\ar@/^/[l]^{r_Z}}\] denote the corresponding inclusion and retraction maps. Since $Z$ is a deformation retract, both $i_Z$ and $r_Z$ are homotopy equivalences. It follows that $i_Z\circ r_Z\circ i_{Y}: Y\to G/\Gamma$ is the desired homotopy equivalence.
\end{proof}
\subsection{Attaining a Correspondence} 
In this section, we define the notion of correspondences between manifolds and construct the correspondences in our place of interest. The notion of correspondence mimics a similar notion in algebraic geometry. However, since the category of manifolds has no fibered products (as opposed to schemes), the definitions here use the Transversality Theorem \cite[Chapter 2 \S3]{gui} to avoid problematic fiber products. Although the definition is natural, we could not find it in the literature in this simplicity. 

We believe that one can replace the category of manifolds with the category of \emph{derived manifolds}, see \cite{spivak}. This category includes all manifolds and has fibered products, and requires understanding of higher categories. 
Another possible approach is to mimic techniques used for symplectic manifold, see \cite{Weinstein}. In this setting, one encounters a similar transversality issue, and several methods are utilized to overcome it. 
In order to keep this manuscript accessible to readers without additional background, we decided to stick with the classical approach and use the following construction.


\label{sub: attaining_a_correspondence}
\begin{de}[Correspondences]
\label{def: correspondence}
	Let $X,Y$ be orientable manifolds without boundary. A {\em correspondence} $f: X\to Y$ is a function-like object specified by the data 
	\[\xymatrix{M\ar[r]^{\xi}\ar[d]^\tau&Y\\X}\]
	where the map $\tau$ is equidimensional and proper, and $M$ is a boundaryless orientable manifold. 
	The {\em multiplicity} of the correspondence is the degree of $\tau$.
	For every $x\in X$ define the image $f(x)\subset Y$ as the set $\xi(\tau^{-1}(x))$. 
\end{de}
\begin{discussion}[Composition is not straightforward]\label{rem: composition}
One would like to consider the composition 
$\xymatrix{X\ar[r]^{f_1}& Y\ar[r]^{f_2}&Z}$ with \[f_1=(M_1, \tau_1, \xi_1),\quad f_2 = (M_2, \tau_2, \xi_2).\]
To this end, it would be natural to take the fibered product 
$M_1\times_YM_2$ 
and include it in the diagram
\[\xymatrix{M_1\times_YM_2\ar@{-->}[r]\ar@{-->}[d]&M_2\ar[r]^{\xi_2}\ar[d]^{\tau_2}&Z\\
M_1\ar[r]^{\xi_1}\ar[d]^{\tau_1}&Y\\
X}\]
Unfortunately, $M_1\times_Y M_2$ may not be a manifold, as required by the definition of a correspondence. 
To deal with this issue we need $\xi_1$ and $\tau_2$ to be transverse, a property we obtain by applying the Transversality Theorem. 
Since neither $\xi_1$ nor $\tau_2$ is an immersion, the Transversality Theorem cannot be applied directly, which leads us to the following definition.
\end{discussion}
\begin{de}[Composition]
\label{de: composition}
Let $\xymatrix{X\ar[r]^{f_1}& Y\ar[r]^{f_2}&Z}$ be two correspondences with the correspondence data
\[f_1=(M_1, \tau_1, \xi_1),\quad\text{and}\quad f_2 = (M_2, \tau_2, \xi_2).\]
Factor $\xi_1: M_1\to Y$ as the composition of an immersion and a product map
\[\xymatrix{M_1 \ar[r]^(.4){\tilde \xi_1} & M_1\times Y\ar[r]& Y}\]
and extend the diagram accordingly:
\begin{align}\label{eq: big diagram}
\xymatrix{M_1\times_YM_2\ar@{-->}[r]\ar@{-->}[d]&M_1\times M_2\ar[r]\ar[d]^{{\rm Id}_{M_1}\times \tau_2}&M_2\ar[r]^{\xi_2}\ar[d]^{\tau_2}&Z\\
M_1\ar[r]^(.45){\tilde \xi_1}\ar[d]^{\tau_1}&M_1\times Y\ar[r]	&Y\\
X}
\end{align}
We will now forget about the $M_2\to Y$ column and define $\tilde \xi_2: M_1\times M_2\to Z$ as the composition of the projection $M_1\times M_2\to M_2$ and $\xi_2$:
\begin{align*}
\xymatrix{M_1\times_YM_2\ar@{-->}[r]\ar@{-->}[d]&M_1\times M_2\ar[r]^(.65){\tilde \xi_2}\ar[d]^{{\rm Id}_{M_1}\times \tau_2}&Z\\
M_1\ar[r]^(.45){\tilde \xi_1}\ar[d]^{\tau_1}&M_1\times Y\\
X}
\end{align*}
By the Transversality Theorem, there exists a proper homotopy between ${\rm Id}_{M_1}\times \tau_2$ and some map $\rho: M_1\times M_2\to M_1\times Y$ which is transverse to $\tilde \xi_1$. Explicitly, there exists a proper map $h:(M_1\times M_2)\times[0,1]$ such that $h(-, 0) = {\rm Id}_{M_1}\times \tau_2$ and $h(-,1)=\rho$ is transverse to $\tilde \xi_1$.

Construct $M_3=M_1\times_{M_1\times Y}(M_1\times M_2)$ to be the fibered product with the diagram 
\begin{align}\label{eq: M3 product diagram}
\xymatrix{M_3\ar[r]\ar[d]&M_1\times M_2\ar[d]^{\rho}\\
M_1\ar[r]^{\tilde \xi_1}&M_1\times Y	}.
\end{align}
The transversality of $\rho$ and $\tilde \xi_1$ implies that $M_3$ is a manifold. Define $\xi_3$ to be the composition of the projection $M_3\to M_1\times M_2$ and $\tilde \xi_2$. 
Define $\tau_3$ to be the composition of the projection $M_3\to M_1$ and $\tau_1$. 
The construction guarantees that each of the maps going down in Diagrams \eqref{eq: big diagram} and \eqref{eq: M3 product diagram} are proper and have domain and codomain of equal dimensions. Hence $f_3 = (M_3, \tau_3, \xi_3)$ is a correspondence $f_3: X\to Z$. 

The data of $f_3$ and $h$ is called {\em composition data} and $f_3$ is the composition of the correspondences $f_1$ and $f_2$.

Consider the correspondences \begin{align*}
\tilde f_1 &:= (M_1, \tau_1, \tilde \xi_1): X\to M_1\times Y,\\
\tilde f_{2, t} &:= (M_1\times M_2, h_t, \tilde \xi_2): M_1\times Y\to Z.
\end{align*}
Let $U\subset X\times Z$ be such that 
for every $x\in X, z\in f_2(f_1(x))$ we have $(x,z) \in U$. 
We say that the composition data $(f_3, h)$ is {\em $U$-compatible} if for every $x\in X, t\in [0,1], z\in \tilde f_{2,t}(\tilde f_1(x))$ we have $(x,z)\in U$. 
\end{de}
\begin{remark}[Multiplicity of composition]
Let $\xymatrix{X\ar[r]^{f_1}& Y\ar[r]^{f_2}&Z}$
be two correspondences and $f_3$ a composition correspondence. Then the multiplicity of $f_3$ is the product of the multiplicities of $f_1$  and $f_2$. 
\end{remark}
\begin{remark}[Categorification of correspondences]
The composition of two correspondences is not uniquely defined. Such a phenomenon is common with $\infty$-categories. We believe one can construct an $\infty$-category of manifolds with correspondences. 
\end{remark}
\begin{lem}
In the composition setting as above, $\xymatrix{X\ar[r]^{f_1}& Y\ar[r]^{f_2}&Z,}$ let $U\subset X\times Z$
 be an open set such that for every $x\in X, z\in f_2(f_1(x))$ we have $(x,z) \in U$. 
Then there exists a $U$-compatible composition. 
\end{lem}
\begin{proof}
Using the notation of Definition \ref{de: composition}, the $U$-compatibility condition is equivalent to: for all $m_{12}\in M_1\times M_2$, $x\in X$, and $t\in [0,1]$ 
\[ (x,\tilde \xi_2(m_{12}))\nin U\implies h_t(m_{12})\nin \tilde f_1(x).\]
This is equivalent to 
\[\forall m_{12}\in M_1\times M_2, t\in [0,1],~~ h_t(m_{12})\nin\tilde f_1(X\times \{\tilde \xi_2(m_{12})\} \setminus U).\]
Here we identify $X\times \{\tilde \xi_2(m_{12})\} \setminus U$ with the corresponding subset of $X$.
Note that $\tilde f_1$ maps closed sets to closed sets. Indeed, $\tau_1^{-1}$ maps closed sets to closed sets as it is the inverse of a continuous map and so does $\tilde \xi_1$ since it is a graph map. 

Consequently, there is a closed set $V\subset (M_1\times M_2) \times (M_1\times Y)$ such that 
$(m_{12}, ({\rm Id}_{M_1}\times \tau)(m_{12}))\nin V$, and our goal is find a proper homotopy $h_t$ between ${\rm Id}_{M_1}\times \tau$ and a map that is transverse to $\tilde \xi_1$, and such that for all $t\in [0,1]$ we have $(m_{12}, h_t(m_{12}))\nin V$. 

Examining the construction in the Transversality Theorem \cite[Chapter 2 \S3]{gui} we see that we can add an open condition on $h_t$. 
\end{proof}

Fix a compact set $Y$ and a homotopy equivalence $\xymatrix{Y\ar@/^/[r]^{\xi_1}&G/\Gamma\ar@/^/[l]^{\xi_2}}$ as in Lemma \ref{lem: exists compact subset h-equiv}. 
Denote by $h(-, t):\xi_1\circ \xi_2\sim {\rm Id}_{G/\Gamma}$ the homotopy and let $Y_0={\rm Im}(\xi_1)\subset G/\Gamma$.

\begin{lem}\label{lem:identity homotopy}
There is a unique continuous map $\tilde h: G/\Gamma\times[0,1]\to G$ such that $\tilde h(x,0) = e$ (where $e$ is the identity element in $G$) for every $x\in G/\Gamma$ and for every fixed $t\in [0,1]$ we have $h(x, t) = \tilde h(x, t)x$. 
\end{lem}
\begin{proof}
This follows immediately from the fact that $G$ is a covering space of $G/\Gamma$. 
\end{proof}

\begin{lem}\label{lem: far correspondence}
For every compact set $K\subset G/\Gamma$ there exists a bounded correspondence $\phi_K: G/\Gamma\setminus Y_0 \to G$ of multiplicity $1$ such that for every $x\in G/\Gamma\setminus Y_0$ we have $\left(\phi_K(x)\cdot x \right)\cap K = \emptyset$.
\end{lem}
\begin{proof}
Let $U$ be an open set such that $K \subset U\subset G/\Gamma$ and $U$ has a compact closure and manifold boundary $N = \partial U$. 

Informally, the correspondence applies the homotopy $\tilde{h}$, which satisfies the conclusion of Lemma \ref{lem:identity homotopy}, to $N$. Then, for any point $v\in G/\Gamma$ we consider the set of points $B\subset N$ that hit $v$, and send $v$ to the transformation matrices from $v$ to $B$. 

Formally, let $M_0 := N\times [0,1]$ and $M_1 := G/\Gamma\setminus U$. Define $\tau_0: M_0\to G/\Gamma$ by $\tau_0(x, t) := h(x,t)$ and let $\tau_1: M_1\to G/\Gamma$ be the identity map. 
Note that $\tau_0$ is proper since it is a map on a compact set, and $\tau_1$ is proper since it is the identity map on a closed set.

Glue $M_0$ and $M_1$ at $N\times \{0\}\subset M_0$ and $N\subset M_1$ to $M' := M_0 \sqcup_N M_1$. 
Since $\tau_0$ and $\tau_1$ agree on $N$, we can glue them into a map $\tau': M'\to G/\Gamma$. 
Since we glued proper maps, $\tau'$ is proper. Denote $M = \tau'^{-1}(G/\Gamma\setminus Y_0)$. The boundary of $M'$ is $N\times \{1\}$. Note that $\tau(N\times \{1\}) = h(M,1) \subset Y_0$ and hence $M$ is boundaryless (yet not compact). Now define $\tau := \tau'|_{M}$.

Let $\xi_0: M_0 \to G$ be given by $\xi_0(x, t) := \tilde  h(x, t)^{-1}$ and let $\xi_1: M_1:\to G$ be the constant $\xi_1 \equiv I$. Glue $\xi_0$ and $\xi_1$ into a map $\xi: M\to G$. Note that $\xi_0$ has a bounded image, and hence so does $\xi$. 
By construction $\xi(m)\tau(m) \in G/\Gamma\setminus U$ for every $m\in M$. 

We would like to construct a correspondence from $M,\tau, \xi$. The problem is that $M$, $\tau$ and $\xi$ are continuous,  but not smooth. By its definition, $M$ admits a smooth orientable structure. 
Using the Whitney Approximation Theorem one can approximate $\xi$ and $\tau$ by a smooth $\xi_0: M\to G$ and a smooth proper $\tau_0: M\to G/\Gamma\setminus Y_0$ such that $\xi_0(m)\tau_0(m) \in G/\Gamma\setminus K$ for every $m\in M$.
We can choose the approximation so that $\deg \tau_0 = \deg \tau$ and $\xi_0$ is bounded.

The objects $M, \tau_0$ and $\xi_0$ define a bounded correspondence $\phi_K: G/\Gamma\setminus Y_0 \to G$ such that for every $x\in G/\Gamma$ and $n\in \phi_K(x)$ we have $n x\not \in K$.
To calculate the multiplicity $\mult \phi_K = \deg \tau$, take \[x \in G/\Gamma \setminus (\overline U \cup \tau_0(M_0)),\] which exists since $G/\Gamma$ is not compact and $\overline U \cup \tau_0(M_0)$ is compact. 
Then $\tau^{-1}(x) = \tau_1^{-1}(x) = \{x\}$. Therefore, $\tau$ is a bijection between $G/\Gamma \setminus (\overline U \cup \tau_0(M_0))$ and its inverse image under $\tau$, and hence $\deg \tau = 1$.
\end{proof}
The goal of the current and the preceding subsections is to prove the following corollary.
\begin{cor}\label{cor: the actual correspondence}
Let $f: A\to G$ be a continuous bounded map, and assume $\{f(a)ax: a\in A\}\cap Y_0 =\emptyset$ for some $x\in G/\Gamma$. Then for every compact set $K\subset G/\Gamma$ there exists a bounded correspondence $\tilde f: A\to G$ of multiplicity $1$ such that \[\{nax: a\in A, n\in \tilde f(a)\}\cap K =\emptyset.\]
\end{cor}
\begin{proof}
Let $\phi_K$ be as in Lemma \ref{lem: far correspondence} and let $V_K\subset G$ be a bounded open set containing the image of $\phi_K$.
We can approximate $f$ by a smooth map using the Whitney Approximation Theorem while preserving the open condition $\{f(a)ax: a\in A\}\cap Y_0 =\emptyset$,
and assume from now on that $f$ is smooth.

Consider the map $r: A\to G/\Gamma\setminus Y_0, r(a) = f(a)ax$. We would like to find a composition $g = \phi_K\circ r$. For every $a\in A$ we have $\phi_k(r(a))\cdot r(a) \cap K=\emptyset$. Denote by $U_K\subset A\times G$  the subset of $a, m$ such that $mr(a)\nin K$. We can choose the composition $g$ to be a $U_K$-compatible composition, that is, such that $g(a)r(a) = g(a)f(a)ax \cap K=\emptyset$ for every $a\in A$. Since both $f$ and $g$ are bounded, the pointwise multiplication correspondence $gf$ is bounded as well and is the sought-for correspondence.
\end{proof}

\subsection{Construction of a cover (reprise)} 
\label{sub: the_actual_proof}

We are now ready to prove Theorem \ref{thm: alternative intersection} and Corollary \ref{cor: un-homotopable compact orbit}. 
\begin{de}\label{de: has good restriction}
Let $T$ be the maximal real split torus in $G$, $\Psi$ be a set of characters $\lambda: T\to \RR$, and $A\subset T$ be an $l$ dimensional subgroup. We say that $A$ \emph{has good restrictions with respect to $\Psi$} if for every linearly independent set $\Psi'\subset \Psi$ of size at most $l$ the restrictions of $\Psi'$ to $A$ are also independent.
\end{de}
The definition is inspired by the attempt to control the invariance dimension of intersections as in Assertion \ref{point: linearly independent characters} of Proposition \ref{prop: cover locally finite and char lin indp}. Note that not having good restrictions with respect to  $\Psi$ is a closed condition on $A$ and Zariski closed on $\fa\subset \ft$. 

\begin{proof}[Proof of Theorem \emph{\ref{thm: alternative intersection}}]
Let $A\subset T$ be a subgroup with good restrictions with respect to $\Psi$, the set of functionals defined in Proposition \ref{prop: cover locally finite and char lin indp}. Let $\{f(a)ag_0\Gamma: a\in A\}$ be a translated trajectory that does not intersect a deformation retract $Z\subset G/\Gamma$. It follows from Lemma \ref{lem: exists compact subset h-equiv} that there exist a compact space $Y$ and a homotopy equivalence $f: Y\to G/\Gamma$ such that $Y_0={\rm Im}(f) \subset Z$. 

Let $K_0\subset G/\Gamma$ be the compact set and let $\{U_\fu:\fu\in \calr\}$ be the cover of $G_0 = \{g\in G:\pi(g)\nin K_0\}$ as in \S \ref{sec: the cover}.
The intersection of any $r+1$ sets from $\{U_\fu:\fu\in \calr\}$ is trivial.

By Corollary \ref{cor: the actual correspondence}, there exists a bounded correspondence $\tilde f = (M, \tau, \xi) : A\to G$ of multiplicity $1$ such that for every $a\in A$ and $n\in \tilde f(a)$ we have $nag_0\Gamma\nin K_0$, i.e., $\tilde{f}(a)ag_0\subset G_0$. 
Since $\tilde f$ has a bounded image there is a compact set $K_G\subset G$ such that ${\rm Im}(f)\subset K_G$. 

For every $\fu\in \calr$ denote 
\[U_\fu^{M}:= \{m\in M: \xi(m)\tau(m)x\in U_\fu\}.\]
Since $f(a)ag_0\subset G_0$ for all $a\in A$, it follows that the sets $U_\fu^M$ cover $M$. 
By Corollary \ref{cor: compact trajectory perturbation},
there exists a locally finite cover $\{\widetilde U_{\fu}^{K_G, Ag_0}:\fu\in \calr\}$ of $A$ such that 
\[U_\fu^M \subset \tau^{-1}(\widetilde U_{\fu}^{K_G, Ag}),\]
and
\[
    \widetilde{U}_\fu^{K,Ag_0} = \left\{a\in A\::\:\forall\lambda\in\Psi_\fu,\:\lambda(a)>d_\lambda\right\}\]
for some set of characters $\Psi_{\fu}: T\to \RR$ and constants $d_\lambda:\lambda\in \Psi_{\fu}$. 
Moreover, let $\fv_1,\dots,\fv_k\in\calr$. If  $\bigcap_{i=1}^{k} U_{\fv_i}\ne \emptyset$, then by Proposition \ref{prop: cover locally finite and char lin indp} there exist linearly independent characters $\lambda_1,\dots,\lambda_k\in X(T)$ such that for any $1\le i\le k$, $\lambda_i\in\Psi_{\fv_i}$. By the assumption that $A$ has good restrictions with respect to $\Psi$ it follows that the restrictions of $\lambda_1,\dots,\lambda_k$
to $A$ are linearly independent.
Consequently, $\invdim(\widetilde U_{\fu_1}^{K_G,Ag_0}\cap \widetilde U_{\fu_2}^{K_G,Ag_0}\cap\dots\cap \widetilde U_{\fu_k}^{K_G,Ag_0})\le r-k$, and Corollary \ref{cor: R_n convex cover} lead to a contradiction.\qed
\end{proof}

\begin{remark}[Avoiding the good restrictions assumption]
\label{rem: proof of cor somthing}
The good restrictions assumption on $T$ was used to bound the invariance dimension of the intersection $\widetilde U_{\fu_1}^{K_G,Ag_0}\cap \widetilde U_{\fu_2}^{K_G,Ag_0}\cap\dots\cap \widetilde U_{\fu_k}^{K_G,Ag_0}$. The statement of Theorem \ref{thm: alternative intersection} would hold if we could bound this invariance dimension in any other way. The simplest assumption would be to assume that $U_{\fu_1}^{K_G,Ag_0}$ are always bounded, that is, $\varrho_i(a)\bp_\fv\xrightarrow{a\to \infty} \infty$ for all $\fv\in \calr_i$. 
Another way one can avoid the good restriction assumption is by assuming that the orbit is divergent, replacing the cover 
$\{U_\fu^M:\fu\in \calr\}$ with an cover of the form 
$\{U_\fu^M\cap U_{\fu,f}^M: \fu\in \calr\}$ as in the proof of theorem \ref{thm: divergence classification}, and $\widetilde{U}_\fu^{K,Ag_0}$ by the analogous bordered sets. 

These are the ideas behind Proposition \ref{prop: deformation retract intersection}.
\end{remark}



\begin{proof}[Proof of Corollary \emph{\ref{cor: un-homotopable compact orbit}}]
Prasad and Ragunathan \cite{PR} prove the existence of compact trajectories $Tg_0\Gamma\subset G/\Gamma$. 
Since $Tg_0\Gamma$ is compact, it follows that $T\cap g_0\Gamma g_0^{-1}$ is a full-rank subgroup in $T$. 
By the density of rank-$r$ subgroups of $T\cap g_0\Gamma g_0^{-1}$ we conclude that there exists $A \subset T$ of $\dim A=r$ with $\rk (A\cap T\cap g_0\Gamma g_0^{-1})=r$ and $\fa \in U\subset {\rm Gr}(\ft, r)$, where $U$ is the Zariski open set provided by Theorem  \ref{thm: alternative intersection}. 

Finally, by Theorem \ref{thm: alternative intersection}, every homotopy of the compact orbit $Ag_0\Gamma$ intersects every compact deformation retract of $G/\Gamma$, the existence of which is Theorem \ref{thm: there is compact deformation retract}.
\end{proof}

\section{Example \ref{ex: compact peripheral orbit}} 
\label{sec: counter_example}
In this section, we recall \cite[Example 1]{TW} which shows that Theorem \ref{thm: alternative intersection} does not hold for all $A\subset T$. That is, there exists an orbit $Ax$, where $A\subset T$, $\dim A= \rk_\QQ (\bG)$, and $x\in G/\Gamma$, which can be homotoped away from every compact set using multiplication by group elements.
We further show that these orbits can be chosen to be compact, a fact that was not shown in \cite[Example 1]{TW}, and investigate divergent trajectories in $G/\Gamma$.

\subsection{The space}
Let $B$ be an order in a quaternion algebra that splits over $\RR$, that is, $B\otimes \RR\cong M_2(\RR)$. 
For concreteness, take $B := \ZZ[i, j, k]$ such that $i^2 = -1, j^2 = 3, ij = -ji = k$.  
Let $\tilde \bG = \GL_2(B)$. 
Since $\tilde\bG(\RR)\cong \GL_4(\RR)$, we may define $\bG\subset \tilde\bG$ to be the inverse image of $\SL_4(\RR)$, that is, the unique subgroup $\bG\subset \tilde \bG$ which satisfies $\bG(\RR) = \SL_4(\RR)$.
Since $B$ is a division algebra, the maximal rational torus in $\bG$ is \[
\bS = \left\{\begin{pmatrix}
	t&0\\
	0&t^{-1}
\end{pmatrix}: t\in \QQ\right\}.\]
In particular, $\rk_\QQ(\bG) = 1$. 
In order to describe elements of these groups we will define the isomorphisms explicitly.
Let $\iota: B\otimes\RR \to M_2(\RR)$ be the isomorphism defined by \[
\iota(i):=\begin{pmatrix}
    0 & -1 \\ 1 & 0 
\end{pmatrix},\;\iota(j):=\begin{pmatrix}
    \sqrt{3} & 0 \\ 0 & -\sqrt{3} 
\end{pmatrix},\;\iota(k):=\begin{pmatrix}
    0 & \sqrt{3} \\ \sqrt{3} & 0 
\end{pmatrix}; \]
it induces an isomorphism $\iota:\bG(\RR)\to G:= \SL_4(\RR)$. 

Let $\Gamma := \iota(\bG(\ZZ))$. We will study the space $G/\Gamma$. 
Other objects we use in both examples are the unipotent subgroup
\[\bN = \left\{\begin{pmatrix}
	1&*\\
	0&1
\end{pmatrix}\right\}\subset \bG, \]
and \[N = \iota(\bN(\RR)) = \left\{\begin{pmatrix}
	1&&*&*\\
	&1&*&*\\
	&&1&\\
	&&&1
\end{pmatrix}\right\}\subset \SL_4(\RR).\]
Then, \[
\fn=\operatorname{Lie}(N_\RR)=\left\{\begin{pmatrix}
	0&&*&*\\
	&0&*&*\\
	&&0&\\
	&&&0
\end{pmatrix}\right\}\]
is a unipotent Lie algebra (unique up to conjugation), and its integer points are \[
\fn_\ZZ=\fn\cap\fg_\ZZ=\iota\left(\left\{\begin{pmatrix}
	0&X\\
	0&0
\end{pmatrix}\;:\;X\in B\right\}\right). \]

\subsection{Example \ref{ex: compact peripheral orbit}}
Denote 
\begin{align*}
\fg_\ZZ & := \left\{X\in M_2(\iota(B))\;:\;\operatorname{trace}(X)=0\right\}\subset M_4(\RR),\\
g_s & := \exp(\diag(s, s, -s, -s)),\\
S& := \iota(\bS(\RR)) = \left\{
\pm g_s\::\: s\in\RR\right\}.
\end{align*}

Let \[m := \begin{pmatrix}
    1&&&
  \\&&-1&
 \\&1&&
  \\&&&1
\end{pmatrix}.\]
We are interested in orbits of the form $Smg_s\Gamma$ for some $s\in \RR$. 

\begin{claim}\label{claim: compact orbit}
For all $s\in \RR$ the orbit $Smg_s\Gamma$ is periodic. \end{claim}
\begin{proof}
Note that $u := 2 + \sqrt 3$ is a unit in $\ZZ[\sqrt 3]$. One can verify that for $t_0:=\log(u)$ we have
\[m^{-1}(g_{t_0})m=\diag(u,u^{-1},u,u^{-1}) = \iota\begin{pmatrix}
2+j&0\\0&2-j
\end{pmatrix} \in \Gamma
.\] Since the diagonal group is commutative and invariant under conjugation by $m$, we deduce that
\begin{align}\label{eq: counter periodic}
    \nonumber g_{t_0}mg_s\Gamma& =g_{t_0}mg_sm^{-1}m\Gamma\\
    &=mg_sm^{-1}g_{t_0}m\Gamma\\\nonumber
    &=mg_s\Gamma,
\end{align}
as desired.
\end{proof}

Let $\varrho=\varrho_1: G\rightarrow\GL(V)$ be the unique $\QQ$-fundamental representation defined in \S \ref{sec: compactness criterion} (note that in our setting $r=\rank_\QQ \bG=1$), and recall the definition of $\bp_\fu$ given a unipotent subspace $\fu$ of $\fg$, also defined in \S\ref{sec: compactness criterion}. 
Note that $\varrho(g_s) \bp_\fn = \exp(8s) \bp_\fn$ and $\varrho(m^{-1}g_tm) \bp_\fn = \bp_\fn$ for all $s,t\in \RR$. 
It follows from Proposition \ref{prop: compactness criterion} that $Smg_s\Gamma$ lies outside any compact set for all $s$ small enough. 

\subsection{Divergent trajectories in \texorpdfstring{$G/\Gamma$}{G/g}}
Let $\alpha_1 < \alpha_2 < \alpha_3 < \alpha_4$ be real numbers such that $\alpha_1+\alpha_2+\alpha_3 + \alpha_4 = 0$, and let \[A=\left\{a_s = \exp(s\diag(\alpha_1, \alpha_2, \alpha_3, \alpha_4)): s\in\RR\right\}.\] 
We will investigate the $A$-divergent trajectories and show that a complete characterization of divergent trajectories is not as simple as one might hope and even its dimension depends on $A$.

First, by Theorem \ref{thm: divergence classification}, all $A$-divergent trajectories diverge obviously. 
The proof of Theorem \ref{thm: divergence classification} implies a more restricted form of divergence. Assume that $A\pi(g)$ is a divergent trajectory. 
It follows from the proof of the theorem that there are rational parabolic subgroups $\bP_1, \bP_2\subset \bG$ with corresponding unipotent radicals $\fv_1, \fv_2$ such that $\varrho(a_sg)\bp_{\fv_1}\xrightarrow{s\to \infty} 0$ and $\varrho(a_sg)\bp_{\fv_2}\xrightarrow{s\to -\infty} 0$, where $\varrho=\varrho_1$ is the unique $\QQ$-fundamental representation defined in \S\ref{sec: compactness criterion}. 

We note that we could have alternatively derived the existence of $\bP_1, \bP_2$ using that $\bG$ has $\rank_\QQ \bG=1$ and an analysis of the cusps of $X$. 

Let $U=\RR^4$ denote the standard representation of $G=\SL_4(\RR)$, and let $E_1=\{e_1,\dots,e_4\}$ denote the standard basis of $U$. Then, \[
E_2:=\{e_i\wedge e_j,\quad 1\le i\le j\le 4\}\] 
is a basis for $\bigwedge^2 U$. Since $G$ acts on $U$ by left multiplication, it acts on $\bigwedge^2 U$ by \[
g(u_1\wedge u_2)=(g u_1)\wedge (g u_2). \]

In order to understand how $A$ acts on unipotent radicals, the following claim relates the action of $G$ on unipotent radicals to the one of $\SL_4(\RR)$ on two-dimensional subspaces of $\RR^4$. 

\begin{claim}
\label{claim: simplifying radicals}
There is a map of representations $\left(\bigwedge^2U\right)^{\otimes 4}\to V$ such that $(e_1\wedge e_2)^{\otimes 4}\mapsto\bp_\fn$.
In particular, for any $g\in G$,
\[ \varrho(a_sg)\bp_\fn \xrightarrow{s\to \infty} 0\text{ if and only if }a_sg(e_1\wedge e_2) \xrightarrow{s\to \infty} 0.\]
\end{claim}

\begin{proof}
By the classification of irreducible representations (see \cite{bourbaki 7-9, knapp}), it is enough to show that the highest weights of the two representations in question coincide, and that for some choices of ordering of the simple system $(e_1\wedge e_2)^{\otimes 4}$ and $\bp_\fn$ are highest weight vectors in $\left(\bigwedge^2U\right)^{\otimes 4}, V$, respectively. 

We parametrize the torus by $$\exp(\diag(s_1, s_2, s_3, s_4)),\quad s_1+s_2+s_3+s_4=0,$$ and note that a highest weight of $\bigwedge^2 U$ is $s_1 + s_2$ with the eigenvector $e_1\wedge e_2$. 
By the definition of $\bp_\fn$, it is the highest weight eigenvector in $V$. 
A direct computation shows that both vectors correspond to eigenvalues $\exp(4(s_1 + s_2))$.
The result follows.
\end{proof}

\begin{claim}
\label{claim: algebraic description of divergence}
There is an irreducible algebraic subvariety 
${\rm Gr}_+\subset {\rm Gr}(4,2)$ such that for any $g\in G$, $a_sg(v\wedge w) \xrightarrow{s\to \infty} 0$ if and only if $g\spa(v,w) \in {\rm Gr}_+$ for every linearly independent pair of vectors $v,w\in \RR^4$. The set ${\rm Gr}_+$ is irreducible of dimension $1$ if $\alpha_1+\alpha_4 = \alpha_2+\alpha_3 = 0$ and of dimension $2$ otherwise.
\end{claim}
\begin{remark}
The description of ${\rm Gr}_+$ is simple because of a low dimension phenomenon. For a larger dimension of $B$ the set ${\rm Gr}_+$ need not be irreducible, nor all irreducible components need be of the same dimension.
\end{remark}
\begin{remark}
Claim \ref{claim: algebraic description of divergence} is the only place we use the strict inequality of $\alpha_i$. Had we demanded $\alpha_1 \le \alpha_2 \le \alpha_3\le \alpha_4$, the lemma would remain true except for the cases $\alpha_1 = \alpha_2 < \alpha_3 = \alpha_4$, in which ${\rm Gr}_+$ is a singleton, and the trivial case $\alpha_i = 0$. The proof is slightly more technical and would not be done here.
\end{remark}
\begin{proof}
This claim is a simple case of the analysis done in \cite{S2}. We provide its proof here for the sake of completeness.

Recall that Pl\"ucker coordinates identify between $2$-dimensional subspaces $V_0 := \spa(v,w)\subset \RR^4$ and pure wedges $v\wedge w\in \bigwedge^2 \RR^4$ up to multiplication by a scalar. 
For each $0\le m\le 4$ denote $R^m: = \RR^m\times \{0\}^{4-m}\subset \RR^4$. 
Let $I_2 = \{(i,j):1\le i<j\le 4\}$.
For every pair $(i,j)\in I_2$ let  $$X_{i,j}:=\{\spa(v,w): v\in R^i, w\in R^j\quad\text{are linearly independent}\}.$$
Note that for any $(i,j)\in I_2$ the set $X_{i,j}$ is irreducible, since it is the image under span of the Zariski open set of linearly independent sets. 

By the definition, if $i_1\le i_2$ and $j_1\le j_2$, then $X_{i_1j_1}\subset X_{i_2j_2}$. 
Hence, upon defining the partial order $\trianglelefteq$ on $I_2$ by $(i_1, j_1)\trianglelefteq (i_2, j_2)$ wherever $i_1\le i_2$ and $j_1\le j_2$, we get that $X_{i_1,j_1} \subset X_{i_2,j_2}$ whenever $(i_1, j_1)\trianglelefteq (i_2, j_2)$. 
Note that if $v\in R^{i}$ and $w\in R^j$, then all coordinates of $v\wedge w$ with respect to $E_2$ are zero except perhaps the coefficients of $e_{i'}\wedge e_{j'}$ for $i'\le i$ and $j'\le j$. Moreover, Gauss elimination implies that $v\wedge w\in X_{i_0,j_0}$, where $(i_0,j_0)$ is the largest tuple (with respect to $\trianglelefteq$) such that the coefficient of $e_{i_0}\wedge e_{j_0}$ in $v\wedge w$ is non-zero. 
In particular, $a_s (v\wedge w)\xrightarrow{t\to \infty}0$ if and only if $\spa(v,w)\in X_{i,j}$ for some $i<j$ for which $\alpha_i + \alpha_j < 0$. 

Now we will do some combinatorics on tuples. Denote \[I_+:=\{(i,j): i < j \text{ and } \alpha_i+\alpha_j < 0\}.\] 
We claim that $I_+$ admits a maximum with respect to $\trianglelefteq$. 
Indeed, the only non-comparable tuples with respect to $\trianglelefteq$ are $(1,4)$ and $(2,3)$, and they cannot be both in $I_+$, since this would contradict the fact $(\alpha_1 + \alpha_4) + (\alpha_2 + \alpha_3)  = 0$. Denote by $(i_1, j_1)$ the maximum of $I_+$.
We conclude that $a_s (v\wedge w)\xrightarrow{t\to \infty}0$ if and only if $\spa(v,w)\in X_{i_1,j_1}$.

Since $(1,4)$ and $(2,3)$ are the only non-comparable elements in $I_2$, and not both in $I_+$ it follows that either $(i_1, j_1) = (1,4)$, or $(i_1, j_1) = (2,3)$, or both $(1, 4), (2,3)\nin I_+$ and then $\alpha_1 + \alpha_4 = \alpha_2 + \alpha_3 = 0$ and $(i_1, j_1) = (1, 3)$. 
Hence to show ${\rm Gr}_+ = X_{i_1,j_1}$ is as asserted in the claim, we need to compute its dimension and prove that it equals $i_1 + j_1 - 3$. 
The dimension of $$\{(v,w) \in R^{i_1}\times R^{j_1}: v,w \text{ are linearly independent}\}$$ is $i_1+j_1$. 
the dimension of the fibers of the map $\spa$ is $3$. Therefore, the dimension of $X_{i_1,j_1}$ is $i_1+i_2-3$, as desired. 
\end{proof}
Construct the set ${\rm Gr}_+$ as in Claim \ref{claim: algebraic description of divergence} and ${\rm Gr}_-$ for the inverse flow, that is, $$a_s (v\wedge w) \xrightarrow{s\to -\infty}0 \iff \spa(v,w)\in {\rm Gr}_-.$$
Let $$V_G:=\{g\in G: g\spa(e_1, e_2) \in {\rm Gr}_+\text{ and }g\spa(e_3, e_4) \in {\rm Gr}_-\}\subset G.$$
\begin{claim}
The dimension of $V_G$ is $7 + 2\dim {\rm Gr}_+$. 
\end{claim}
\begin{proof}
We will compute the dimension of $V_G$ as a fibration.
Let $${\rm Gr}_\pm := \{(V,W)\in {\rm Gr}_+\times {\rm Gr}_-: V\cap W=\{0\}\}.$$
This set is nonempty since it contains $(\spa(e_1, e_2),\spa(e_3, e_4))$. Its dimension is $\dim {\rm Gr}_+ + \dim {\rm Gr}_- = 2\dim {\rm Gr}_+$.
Note that the map $V_G \to {\rm Gr}_\pm$ defined by $$g\mapsto (g\spa(e_1, e_2),g\spa(e_3, e_4))$$ is onto, and the dimensions of the fibers are equal to the dimension of the stabilizer of a pair of two $2$-dimensional vector subspaces of $\RR^4$, which is $7$. 
The result follows.
\end{proof}
We can now prove the following characterization.

\begin{prop}\label{prop: example divergent classification}
Every $A$ divergent trajectory is of the form $A\pi(g_Xg_\QQ)$, for $g_X\in V_G$ and $g_\QQ\in G(\QQ)$, and any such trajectory is divergent.
\end{prop}
\begin{proof}
Let $a_s\pi(g)$ be a divergent trajectory. By Theorem \ref{thm: divergence classification} (although in the rank-$1$ case it is not needed) there are unipotent radicals $\fv_1, \fv_2$ such that $\varrho(a_sg)\bp_{\fv_1}\xrightarrow{s\to \infty}0$ and $\varrho(a_sg)\bp_{\fv_2}\xrightarrow{s\to -\infty}0$.

Let $P$ be the unique $\QQ$-parabolic subgroup of $G$ that contains $S$ and such that $\fn$ is the Lie algebra of its unipotent radical. By \cite[Proposition 14.21(i)]{borel}, there exists a unique $\QQ$-parabolic subgroup $S\subset P^-\subset G$ which is opposite to $P$. Let $\fn^-$ be the Lie algebra of the unipotent radical of $P^-$.

It follows from \cite[Proposition 14.21(i)]{borel} that there is $g_\QQ\in \bG(\QQ)$ such that $\Ad(g_\QQ) \fn = \fv_1$ and $\Ad(g_\QQ)\fn^- = \fv_2$ (we encourage the reader to prove this statement in our setting). Claim \ref{claim: simplifying radicals}, implies that $a_sg g_\QQ(e_1\wedge e_2) \xrightarrow{s\to \infty}0$. Similarly, $a_sg g_\QQ(e_3\wedge e_4) \xrightarrow{s\to -\infty}0$. 
Therefore, $g_X := g g_\QQ\in V_G$. The first claim follows. 

On the other hand, for any $g_X\in V_G$ and $g_\QQ\in G(\QQ)$, the trajectory $a_s\pi(g_X g_\QQ)$ shrinks the rational vector $\varrho(g_\QQ^{-1}) \bp_\fn$ as $s\rightarrow\infty$ and shrinks $\varrho(g_\QQ^{-1})\bp_{\fn^-}$ as $s\rightarrow-\infty$. Thus, the trajectory $A\pi( g_X g_\QQ)$ diverges.
\end{proof}


\begin{thebibliography}{1}
\bibitem{Ash}
A. Ash, \textit{Small-dimensional classifying spaces for arithmetic subgroups of general linear groups}, Duke Math. J. \textbf{51} (1984), 459--468.

\bibitem{borel2}A. Borel, \emph{Density and maximality of arithmetic subgroups}, Journal Reine Angew. Math. \textbf{224} (1966), 78--89.


\bibitem{borel} A. Borel, \emph{Linear Algebraic Groups}, second enlarged edition, Springer (1991), 55--151.

\bibitem{BT}
A. Borel and J. Tits, \textit{Groupes Reductifs}, Publ. IHES \textbf{27} (1965).

\bibitem{bourbaki 7-9} N. Bourbaki, \emph{Lie Groups and Lie Algebras, Chapters 7--9}, vol. 3, Springer (2008). 

\bibitem{BT2} 
R. Bott and L. W. Tu, \textit{Differential Forms in Algebraic Topology}, vol. 82, Springer (1982).

\bibitem{BS}
A. Borel and J. P. Serre, \textit{Corners and Arithmetic Groups}, Comment. Math. Helv. 48.1 (1973), 436--491.

\bibitem{C2}
Y. Cheung, \textit{Hausdorff Dimension of the Set of Singular Pairs}, Ann. of Math. (2011), 127--167.

\bibitem{dani} G. Dani, \emph{Divergent Trajectories of Flows on Homogeneous Spaces and Diophantine Approximation}, Journal Reine Angew. Math. \textbf{359} (1985), 55--89. 




\bibitem{gui}
V. Guillemin and A. Pollack, \textit{Differential Topology}, vol. 370, American Mathematical Soc., 2010.

\bibitem{gr1}
D. R. Grayson, \textit{Reduction Theory Using Semistability}, Comment. Math. Helv. \textbf{59} (1984), 600--634.

\bibitem{gr2}
D. R. Grayson, \textit{Reduction Theory Using Semistability, II}, Comment. Math. Helv. \textbf{61} (1984), 661--676.

\bibitem{hatcher}
A. Hatcher, \emph{Algebraic Topology}, Cambridge University Press, Cambridge, 2002.

\bibitem{Hattori}T. Hattori, \emph{Geometric limit sets of higher rank lattices}, Proc. London Math. Soc. (3), 90 (2005), no. 3, 689–710.


\bibitem{Littlewood} M. Einsiedler, A. Katok, and E. Lindenstrauss, \emph{Invariant Measures and the Set of Exceptions to Littlewood's Conjecture}, Ann. of Math. \textbf{164} (2006), 513--560. 

\bibitem{singular system} S. Kadyrov, D. Kleinbock, E. Lindenstrauss, and G. Margulis, \emph{Singular Systems of Linear Forms and Non-Escape of Mass in the Space of Lattices}, J. Anal. Math. \textbf{133} (2017), 253--277. 

\bibitem{number theory} D. Kleinbock, N. Shah, and A. Starkov, \emph{Dynamics of Subgroup Actions on Homogeneous Spaces of Lie Groups and Applications to Number Theory}, Handbook of Dynamical Systems, \textbf{1}, North-Holland, Amsterdam (2002), 813--930.

\bibitem{KleinbockWeiss} D. Kleinbock and B. Weiss, \emph{Modified Schmidt Games and a Conjecture of Margulis}, J. Mod. Dyn. \textbf{7} (2013) 429--460.

\bibitem{knapp}A. W. Knapp, \emph{Lie Groups Beyond an Introduction}, 2nd ed., Birkh\"auser, Boston, 1996.

\bibitem{Lindenstrauss notes}E. Lindenstrauss, \textit{Recent Progress on Rigidity Properties of Higher Rank Diagonalizable Actions and Applications}, A survey to appear in a volume dedicated to G. A. Margulis, arXiv:2101.11114. 

\bibitem{Maucourant}F. Maucourant, \emph{A Nonhomogeneous Orbit Closure of a Diagonal Subgroup}, \textbf{171} (2010), Issue 1, 557--570. 

\bibitem{Mc} 
	C. T. McMullen, {Minkowski’s Conjecture, Well-Rounded Lattices and Topological Dimension}, \textit{J. Amer. Math. Soc.} {\bf 18} (2005), no. 3, 711--734, (electronic).

\bibitem{PR}
G. Prasad and M. S. Ragunathan, 
\emph{Cartan Subgroups and Lattices in Semi-Simple Groups},
Annals of Mathematics, \textbf{96}(2) (1972), 296--317.

\bibitem{Raghu}
M. S. Raghunathan, \textit{A note on quotients of real algebraic groups by arithmetic subgroups}, Invent. Math. \textbf{4}(5) (1968), 318--335.

\bibitem{ratner} M. Ratner, \emph{On Measure Rigidity of Unipotent Subgroups of Semisimple Groups}, Acta Math. \textbf{165} (1) (1990), 229--309. 

\bibitem{PS}
A. Pettet and J. Souto, \textit{Periodic Maximal Flats Are Not Peripheral}, Journal of Topology \textbf{7}(2) (2009)


\bibitem{saper}
L. Saper, \textit{Tilings and finite energy retractions of locally symmetric spaces}, Comment. Math. Helv. \textbf{30} (1997), 167--202.

\bibitem{shapira} U. Shapira, \emph{Full Escape of Mass for the Diagonal Group}, International Mathematics Research Notices, \textbf{15} (2017), 4704--4731.

\bibitem{shap_weiss}
U. Shapira, B. Weiss, \textit{A Volume Estimate for the Set of Stable Lattices}." Comptes Rendus Mathématique \textbf{352} (2014), 875--879.

\bibitem{S}
O. N. Solan, \textit{Stable and Well-Rounded Lattices in Diagonal Orbits}, Isr. J. Math. \textbf{234}, 501--519 (2019).

\bibitem{S2}
O. N. Solan,
\textit{Parametric Geometry of Numbers for a General Flow},
unpublished (2021).

\bibitem{spivak}
D. I. Spivak, \textit{Derived smooth manifolds}, Duke Math. J. \textbf{153}, 55-128 (2010).
\bibitem{tamam} N. Tamam, \emph{Divergent Trajectories in Arithmetic Homogeneous Spaces of Rational Rank Two},
to appear in Ergodic Theory and Dynamical Systems.

\bibitem{tamam2} N. Tamam, \emph{Existence of Non-Obvious Divergent Trajectories in Homogeneous Spaces}, arXiv: 1909.09205.

\bibitem{tits} J. Tits, \emph{A Local Approach to Buildings}, Springer (1981), 519--547. 

\bibitem{T S-adic} G. Tomanov, \emph{Values of decomposable forms at $S$-integral points and orbits of tori on homogeneous spaces}, Duke Math. J. 138 (3) (2007) 533 -- 562. 

\bibitem{T} G. Tomanov, \emph{Closures of Locally Divergent Orbits of Maximal Tori and Values of Homogeneous Forms}, Ergodic Theory and Dynamical Systems \textbf{41}(10) (2021) 3142-–3177.

\bibitem{TW} G. Tomanov, B. Weiss, \emph{Closed Orbits for Actions of Maximal Tori on Homogeneous Spaces}, Duke Math. J. \textbf{199}(2) (2003) 367--392. 

\bibitem{LU decom} L. N. Trefethen, D. Bau, \emph{Numerical Linear Algebra}, vol. 50, Siam. 

\bibitem{Weinstein}A. Weinstein, \emph{Symplectic categories}, Portugaliae Mathematica 67.2 (2010): 261-278.

\bibitem{Weiss}B. Weiss, \emph{Divergent Trajectories on Noncompact Parameter Spaces}, Geom. Funct. Anal. \textbf{14}(1) (2004), 94--149.

\bibitem{Weiss2}B. Weiss, \emph{Divergent Trajectories and $\QQ$-Rank}, Israel J. Math. \textbf{152} (2006), 221--227. 

\end{thebibliography}
\end{document}